\documentclass[preprint,10pt]{elsarticle}
\usepackage[utf8]{inputenc}
\usepackage[T1]{fontenc}
\usepackage{lmodern}
\usepackage{geometry}
\usepackage{amsmath,amssymb,amsthm,mathtools}
\usepackage{enumitem}
\usepackage{hyperref}
\hypersetup{colorlinks=true,linkcolor=blue,citecolor=blue,urlcolor=blue}
\usepackage{microtype}
\usepackage{booktabs}
\newtheorem{theorem}{Theorem}[section]
\newtheorem{assumption}[theorem]{Assumption}
\newtheorem{definition}[theorem]{Definition}
\newtheorem{lemma}[theorem]{Lemma}
\newtheorem{corollary}[theorem]{Corollary}
\newtheorem{remark}[theorem]{Remark}

\newcommand{\R}{\mathbb{R}}
\newcommand{\T}{\mathbb{T}}
\newcommand{\Z}{\mathbb{Z}}
\newcommand{\per}{\mathrm{per}}
\newcommand{\e}{\mathrm e}
\newcommand{\unif}{\mathrm{unif}}
\newcommand{\loc}{\mathrm{loc}}
\newcommand{\diverg}{\operatorname{div}}
\newcommand{\osc}{\operatorname*{osc}}

\newcommand{\eps}{\varepsilon}

\newcommand{\avg}[1]{\langle #1\rangle}
\newcommand{\Adiv}{A^{\mathrm{div}}}
\newcommand{\Ahdiv}{\widehat A^{\mathrm{div}}}
\newcommand{\aper}{a^{\per}}
\newcommand{\bper}{b^{\per}}
\newcommand{\mper}{m^{\per}}
\newcommand{\arme}{a^{\rm e}}
\newcommand{\brme}{b^{\rm e}}
\newcommand{\haper}{\widehat{a}^{\rm per}}
\newcommand{\hae}{\widehat{a}^{\rm e}}
\newcommand{\hbe}{\widehat{b}^{\rm e}}

\newcommand{\hm}{\widehat{m}}

\newcommand{\curl}{\operatorname{curl}}

\journal{Journal de Math\'ematiques Pures et Appliqu\'ees}

\begin{document}

\begin{frontmatter}

\title{A two-dimensional structural local-defect theory for scalar non-divergence\\
advection--diffusion homogenization}

\author[a,b]{Jizu Huang}
\author[a,b]{Yong Ma\corref{cor1}}
\ead{mayong@amss.ac.cn}
\cortext[cor1]{Corresponding author.}
\address[a]{SKLMS, Academy of Mathematics and Systems Science,
Chinese Academy of Sciences, Beijing 100190, P.R. China}

\address[b]{School of Mathematical Sciences,
University of Chinese Academy of Sciences, Beijing 100190, P.R. China}
\begin{abstract}
We develop a two-dimensional structural local-defect theory for scalar non-divergence
advection--diffusion operators
\[
 Lu=-a:D^2u+b\cdot\nabla u,
 \qquad a=\aper+\arme,
 \qquad b=\bper+\brme.
\]
The periodic coefficients and the bounded local defects are uniformly H\"older continuous,
the interpolating matrices \(a_t=\aper+t\arme\) are symmetric and uniformly elliptic, and
\(\arme\in L^r(\R^2)\), \(\brme\in L^s(\R^2)\), where \(1<r,s<2\). Under the periodic
centering condition \(\avg{\mper\bper}=0\), global harmonic coordinates remove the periodic
drift. After blow-down, the transformed defect drift is small in the critical local \(L^2\)
space. Critical-drift compactness then yields a finite-energy Liouville theorem and closes the
continuation argument, giving a whole-space estimate for \(1<q<2\) and
\(1/q^*=1/q-1/2\). This estimate provides defect correctors and, by duality, a positive
invariant density \(m=\mper+m^{\e}\). A planar Hodge construction in harmonic coordinates,
followed by a Piola pull-back, produces a skew-symmetric field \(B=B^{\per}+B^{\e}\) such that
\[
 mLu=-\diverg\bigl((ma-B)\nabla u\bigr).
\]
If \(M=\max\{r,s\}\) and \(M^*=2M/(2-M)\), then, for some \(\beta>0\), the defects
\(B^{\e}\) and \(A^{\e}:=ma-B-\Adiv_{\per}\) belong to
\(L^{M^*}\cap L^\infty\cap C^{0,\beta}_{\unif}\) and vanish uniformly at infinity. This
supplies the missing two-dimensional structural reduction in the scalar regular non-endpoint
regime.
\end{abstract}

\begin{keyword}
homogenization \sep local defects \sep advection--diffusion \sep invariant measure \sep
Liouville theorem \sep harmonic coordinates \sep Hodge decomposition \sep Piola transform
\MSC[2020] 35B27 \sep 35J15 \sep 35J70 \sep 35B45 \sep 42B20
\end{keyword}

\end{frontmatter}

\section{Introduction}

We consider scalar non-divergence advection--diffusion operators in dimension two,
\begin{equation}\label{eq:intro-operator}
 Lu=-a:D^2u+b\cdot\nabla u,
 \qquad a=\aper+\arme,
 \qquad b=\bper+\brme.
\end{equation}
The periodic coefficients belong to \(C^{0,\tau}_{\per}(\R^2)\), while the bounded local
defects are uniformly H\"older continuous and satisfy
\[
 \arme\in L^r(\R^2),\qquad \brme\in L^s(\R^2),\qquad 1<r,s<2.
\]
For \(t\in[0,1]\), we set
\[
 L_t=-a_t:D^2+b_t\cdot\nabla,\qquad
 a_t=\aper+t\arme,\qquad b_t=\bper+t\brme,
\]
and assume that the symmetric matrices \(a_t\) are uniformly elliptic, independently of \(t\).
Writing \(\mper\) for the normalized positive periodic adjoint density, our structural
compatibility condition is
\begin{equation}\label{eq:intro-centering}
 \avg{\mper\bper}=0.
\end{equation}

Periodic homogenization is classical; see, for example,
\cite{AvellanedaLin,BLP}. Blanc, Le Bris and Lions developed the local-defect program for
divergence-form equations and local profiles \cite{BLLMilan,BLLCRAS,BLLLocalProfiles}, and then
for non-divergence equations without drift \cite{BLLNoDrift}. In the divergence-form setting,
Blanc, Josien and Le Bris subsequently obtained precise defect-adapted approximations, including
\(W^{1,p}\) and Lipschitz-type estimates with explicit rates under the hypotheses of their work
\cite{BJL2020}. That quantitative theory explains the value of reaching a localized-periodic
divergence-form coefficient class, but no quantitative two-scale estimate is proved in the present
paper.

For advection--diffusion operators, Blanc, Le Bris and Lions used an invariant density to obtain
the required divergence-form structure in dimensions \(d\ge3\) \cite{BLL,BLLErratum}. We use
throughout the corrected non-endpoint formulation recorded in their erratum, including the
centering condition \eqref{eq:intro-centering}. With the sign convention in
\eqref{eq:intro-operator}, the invariant equation is
\[
 -\partial_i\bigl(\partial_j(a_{ij}m)+b_i m\bigr)=0.
\]
If a skew-symmetric matrix field \(B\) satisfies
\[
 \diverg B=mb+\diverg(ma),
\]
then
\begin{equation}\label{eq:intro-reduction}
 mLu=-\diverg\bigl((ma-B)\nabla u\bigr).
\end{equation}
This invariant-measure identity is the structural bridge from the non-divergence problem to the
divergence-form local-defect setting.

The higher-dimensional proof does not extend directly to the plane. There are two distinct
obstructions. First, the exponent range used in the closedness step of the whole-space
continuation argument becomes empty when \(d=2\). Second, the defect part of \(B\) is represented
in two dimensions by a scalar stream function, whose first-order potential has borderline
far-field behaviour. The present paper resolves these two planar problems separately.

The first mechanism is a geometric change of variables. Condition \eqref{eq:intro-centering} is
exactly the Fredholm compatibility condition for the periodic cell problems
\[
 P_\alpha(x)=x_\alpha+\chi_\alpha(x),\qquad L_{\per}P_\alpha=0,\qquad \alpha=1,2.
\]
The periodic \(\sigma\)-harmonic mapping theory and the planar Lewy-type nonvanishing theorem of
Alessandrini and Nesi imply that \(P=x+\chi\) is a global uniformly bi-Lipschitz diffeomorphism
\cite{AlessandriniNesi1,AlessandriniNesi2}. In the variables \(Y=P(x)\), the periodic drift
vanishes. The remaining drift is
\[
 \widehat b^{\e}_{\alpha}(Y)
 =-\arme(x):D^2P_\alpha(x)+\brme(x)\cdot\nabla P_\alpha(x),
 \qquad x=P^{-1}(Y),
\]
and belongs to \(L^M(\R^2)\cap L^\infty(\R^2)\), where
\[
 M=\max\{r,s\}\in(1,2),\qquad M^*=\frac{2M}{2-M}>2.
\]
After blow-down, its rescaled \(L^2\)-norm tends to zero on fixed annuli. Mooney's
critical-drift estimates, applied only to strong \(W^{2,2}\) solutions, provide the compactness
needed for a finite-energy Liouville theorem \cite{Mooney}. This replaces the missing planar
closedness argument and yields, for \(1<q<2\) and \(1/q^*=1/q-1/2\),
\begin{equation}\label{eq:intro-central}
 \|\nabla u\|_{L^{q^*}(\R^2)}+\|D^2u\|_{L^{q^*}(\R^2)}
 \le C_q\bigl(\|f\|_{L^q(\R^2)}+\|f\|_{L^{q^*}(\R^2)}\bigr),
\end{equation}
uniformly along the family \(L_tu=f\).

Estimate \eqref{eq:intro-central} gives defect correctors and, by duality, an invariant density
\(m=\mper+m^{\e}\). The density \(m\) is positive and bounded away from zero, while
\(m^{\e}(x)\to0\) uniformly as \(|x|\to\infty\). Its local boundedness and uniform H\"older
regularity follow from the double-divergence/Fokker--Planck theory of
Bogachev--Shaposhnikov \cite{BogachevShaposhnikov}; positivity and the comparison argument use
Bauman's adjoint maximum principle \cite{Bauman}.

The second mechanism is a planar Hodge construction in harmonic coordinates. The defect source
has the form \(H+\diverg Q\), with \(H\in L^M\cap L^\gamma\) for every finite \(\gamma>2\) and
\(Q\in L^{M^*}\cap L^\infty\cap C^{0,\theta}_{\unif}\). A first-order potential controls the
\(H\)-part, while cancellation in a zero-order Calder\'on--Zygmund operator controls the
\(Q\)-part. A Piola pull-back produces \(B=B^{\per}+B^{\e}\) and the identity
\eqref{eq:intro-reduction}. For some \(\beta>0\),
\[
 B^{\e},A^{\e}\in L^{M^*}(\R^2)\cap L^\infty(\R^2)
 \cap C^{0,\beta}_{\unif}(\R^2),
 \qquad A^{\e}=m^{\e}\aper+m\arme-B^{\e},
\]
and both defects vanish uniformly at infinity. Since \(B\) is skew-symmetric,
\[
 \xi\cdot (ma-B)\xi=m\,\xi\cdot a\xi;
\]
hence the divergence-form coefficient is coercive by the positive lower bound for \(m\) and the
uniform ellipticity of \(a\).

The contribution is structural: it supplies the two-dimensional reduction to a coercive
divergence-form coefficient which is a localized perturbation of a periodic one. Applying any
particular downstream homogenization theorem still requires its remaining hypotheses. Related
developments for almost translation-invariant coefficients, nonlinear \(p\)-Laplace equations,
and variational closure under nonperiodic perturbations can be found in
\cite{BraidesDalMasoLeBris,Goudey,Wolf}. Section~2 states the assumptions and main results.
Section~3 constructs harmonic coordinates and proves the Liouville theorem, central estimate and
corrector result. Section~4 constructs the invariant density. Section~5 proves the Hodge estimate
and the Piola reduction.

\section{Assumptions and main results}

We work in dimension two. Repeated indices are summed. For a matrix field $Q$,
$(\diverg Q)_j=\partial_iQ_{ij}$. We identify $\T^2=\R^2/\Z^2$.

Throughout the paper, hatted quantities are in harmonic coordinates $Y=P(x)$, while unhatted
quantities are in the original variables $x$. We denote the inverse harmonic-coordinate map by
$\Theta=P^{-1}$, so that $x=\Theta(Y)$. The symbols $\Adiv_{\per}$ and $\Ahdiv_{\per}$ denote,
respectively, the original and transformed periodic divergence-form representatives, whereas
$A=m\,a-B$ denotes the final full divergence-form coefficient. The letter $Q$ is reserved for
matrix fields in Hodge-type decompositions, such as $F=H+\diverg Q$.

\medskip

For the reader's convenience, the main notational conventions are summarized below.

\begin{center}
\begin{tabular}{ll}
\toprule
\textbf{Symbol} & \textbf{Meaning} \\
\midrule
\(x\) & original coordinate \\
\(Y=P(x)\) & harmonic coordinate \\
\(\Theta=P^{-1}\) & inverse harmonic-coordinate map \\
\(\widehat{\cdot}\) & quantity written in the \(Y\)-coordinates \\
\(\widehat a_{\rm per}\) & transformed periodic non-divergence coefficient \\
\(\hae\) & transformed local defect of the second-order coefficient \\
\(\hbe\) & transformed local defect drift in harmonic coordinates \\
\(A_{\rm per}^{\rm div}\) & original periodic divergence-form representative \\
\(\widehat A_{\rm per}^{\rm div}\) & transformed periodic divergence-form representative \\
\(B_{\rm per}\) & original periodic skew potential \\
\(\widehat B_{\rm per}\) & transformed periodic skew potential \\
\(m=\mper+m^{\e}\) & full invariant density in the original variables \\
\(\hm=\widehat m^{\per}+\widehat m^{\e}\) & push-forward invariant density in harmonic coordinates \\
\(B=B^{\per}+B^{\e}\) & full skew potential in the original variables \\
\(A=ma-B\) & final full divergence-form coefficient \\
\(Q\) & generic matrix field used in Hodge decompositions, as in \(F=H+\diverg Q\) \\
\bottomrule
\end{tabular}
\end{center}

\medskip

\begin{assumption}[Regular scalar two-dimensional setting]
The periodic coefficients $\aper,\bper$ are $\Z^2$-periodic and belong to
$\left(C^{0,\tau}_{\per}(\mathbb{R}^2)\right)^{2\times 2}$ and $\left(C^{0,\tau}_{\per}(\mathbb{R}^2)\right)^{2}$ for some $\tau\in(0,1)$. The matrices $\aper$ and
\[
 a_t=\aper+t \arme
\]
are symmetric, and there exist constants $0<\lambda\leq\Lambda<\infty$ such that
\[
 \lambda |\xi|^2\leq \xi\cdot a_t(x)\xi\leq \Lambda |\xi|^2
 \qquad \forall x\in\R^2,\ \forall t\in[0,1],\ \forall \xi\in\R^2.
\]
The defects satisfy
\begin{align}
 \arme&\in \left( C^{0,\tau}_{\unif}(\R^2)\right)^{2\times 2}\cap \left(L^\infty(\R^2)\right)^{2\times 2}\cap \left(L^r(\R^2)\right)^{2\times 2},\label{eaass}\\
 \brme&\in \left( C^{0,\tau}_{\unif}(\R^2)\right)^{2}\cap\left( L^\infty(\R^2)\right)^{2}\cap \left(L^s(\R^2)\right)^{2},\label{ebass}
\end{align}
where $1<r,s<2$. We set
\[
 L_tu=-a_t:D^2u+b_t\cdot \nabla u,
 \qquad b_t=\bper+t \brme.
\]
\end{assumption}

\begin{lemma}[Periodic invariant density]\label{lem:per-density}
There is a unique normalized positive periodic adjoint density $\mper$ satisfying
\begin{equation}
 -\partial_i\bigl(\partial_j(a^{\per}_{ij}\mper)+b^{\per}_i\mper\bigr)=0
 \quad\text{on }\T^2,
 \qquad \avg{\mper}=1.
\end{equation}
Moreover $\mper$ is bounded above and below by positive constants and belongs to
$C^{0,\theta_0}_{\per}(\mathbb{T}^2)$ for some $\theta_0\in(0,1)$.
\end{lemma}

\begin{proof}
This is the standard periodic adjoint theory for non-divergence operators with H\"older
coefficients. The Fredholm alternative gives uniqueness up to normalization, and the local
regularity and Harnack principle for double-divergence/Fokker--Planck densities give the stated
bounds; see Bogachev--Shaposhnikov \cite[Theorems~2.1 and~3.1, Corollary~3.6]{BogachevShaposhnikov}.
\end{proof}

\begin{assumption}[Periodic centering]\label{ass:centering}
For the density $\mper$ of Lemma~\ref{lem:per-density}, we assume
\begin{equation}
 \avg{\mper\,\bper}=0.
\end{equation}
\end{assumption}

\begin{definition}
Let
\begin{equation}
 M=\max\{r,s\},\qquad M^*=\frac{2M}{2-M}.
\end{equation}
Then $1<M<2$ and $M^*>2$.
\end{definition}

\begin{theorem}[Harmonic-coordinate gauge]\label{thm:gauge}
Under Assumption~2.1 and the periodic centering assumption~\ref{ass:centering}, there exists
\[
 P(x)=x+\chi(x),\qquad \chi\in (C^{2,\tau'}_{\per}(\R^2))^2,
\]
for some $\tau'\in(0,\tau]$, such that $L_{\per}P_\alpha=0$, $\alpha=1,2$. The map
$P:\R^2\to\R^2$ is a uniformly bi-Lipschitz $C^{2,\tau'}_{\rm loc}$ diffeomorphism, with
$DP$ and $D^2P$ bounded and periodic. If $Y=P(x)$, then, for every $t\in[0,1]$ and every $1<p<\infty$, the operator
$L_t:W^{2,p}_{\rm loc}(\R^2)\rightarrow L^p_{\rm loc}(\R^2)$ can be written as:
\begin{equation}
 L_t =\widehat L_t :=-(\haper+t\hae):D^2_Y+t\hbe\cdot\nabla_Y,
\end{equation}
where,
\begin{align}
 \haper(Y)&=DP(x)\aper(x)DP(x)^T,\label{haperdef}\\
 \hae(Y)&=DP(x)\arme(x)DP(x)^T,\label{haedef}\\
  \hbe_{\alpha}(Y)&=-\arme(x):D^2P_\alpha(x)+\brme(x)\cdot\nabla P_\alpha(x),\quad\alpha=1,\,2.
 \label{hbedef}
\end{align}
The transformed periodic drift is zero, and $\hae_{\alpha\beta},\,\hbe_\alpha\in L^M(\R^2)\cap L^\infty(\R^2)$.
\end{theorem}

\begin{theorem}[Finite-energy Liouville theorem]\label{thm:liouville}
Fix $t\in[0,1]$ and let $p>2$. If
\[
 L_tu=0\quad\text{in }\R^2,
 \qquad u\in W^{2,p}_{\loc}(\R^2),
 \qquad \nabla u\in (L^p(\R^2))^{2},\,D^2u\in (L^p(\R^2))^{2\times 2},
\]
then $u$ is constant.
\end{theorem}

\begin{definition}
    For $f\in L^q(\R^2)\cap L^p(\R^2)$ and $1\leq q< p\leq \infty$, let us define the norm $\|f\|_{(L^q\cap L^p)(\R^2)}$ as
    \begin{equation}
        \|f\|_{(L^q\cap L^p)(\R^2)} = \|f\|_{L^q(\R^2)}+ \|f\|_{ L^p(\R^2)}.
    \end{equation}
\end{definition}

\begin{theorem}[Central estimate]\label{thm:central}
Let $1<q<2$ and define $q^*$ by
\[
 \frac1{q^*}=\frac1q-\frac12.
\]
For every $t\in[0,1]$ and every $f\in L^q(\R^2)\cap L^{q^*}(\R^2)$, the equation
\[
 L_tu=f\quad\text{in }\R^2
\]
has a solution, unique modulo constants, such that $\nabla u\in (L^{q^*}(\R^2))^{2},\,D^2u\in (L^{q^*}(\R^2))^{2\times 2}$ and
\begin{equation}
 \|\nabla u\|_{(L^{q^*}(\R^2))^2}+\|D^2u\|_{(L^{q^*}(\R^2))^{2\times 2}}
 \leq C_q\bigl\|f\|_{(L^q\cap L^{q^*})(\R^2)}.
 \label{centralestimate}
\end{equation}
\end{theorem}

\begin{theorem}[Correctors]\label{thm:correctors}
For every $\xi\in\R^2$ there exists a corrector
\[
 w_\xi=w_{\xi}^{\per}+w^{\e}_\xi
\]
solving
\[
 Lw_\xi=-b\cdot \xi\quad\text{in }\R^2.
\]
Moreover
\[
 \nabla w^{\e}_\xi\in(L^{M^*}(\R^2))^2,\,\quad\ D^2w^{\e}_\xi\in (L^{M^*}(\R^2))^{2\times 2},
 \qquad
 \frac{w^{\e}_\xi(x)}{1+|x|}\xrightarrow{|x|\to\infty}0.
\]
\end{theorem}

\begin{theorem}[Invariant measure]\label{thm:inv}
There exists a positive invariant measure
\[
 m=\mper+m^{\e}
\]
satisfying
\begin{equation}
 -\partial_i\bigl(\partial_j(a_{ij}m)+b_i m\bigr)=0
 \quad\text{in }\R^2.
 \label{adjointfull}
\end{equation}
Furthermore
\[
 m^{\e}\in L^{M^*}(\R^2)\cap L^\infty(\R^2)\cap C^{0,\theta}_{\unif}(\R^2)
 \quad\text{for some }\theta\in(0,1),
 \qquad \lim_{R\to\infty}\sup_{|x|\ge R}|m^{\e}(x)|=0,
 \qquad \inf_{\R^2}m>0.
\]
The uniqueness in this decaying class is proved in Corollary~\ref{cor:unique}.
\end{theorem}

\begin{theorem}[Divergence-form reduction]\label{thm:divred}
Let $m=\mper+m^{\e}$ be any positive invariant measure for $L$ satisfying
\[
 m^{\e}\in L^{M^*}(\R^2)\cap L^\infty(\R^2),
 \qquad \inf_{\R^2}m>0.
\]
There exists a skew-symmetric matrix
\[
 B=B^{\per}+B^{\e},\qquad B^T=-B,
\]
such that
\begin{equation}
 \diverg B=m\,b+\diverg(m\,a)
 \label{skewsource}
\end{equation}
and, for some \(\beta\in(0,1)\),
\[
 B^{\e}\in (L^{M^*}(\R^2))^{2\times 2}\cap (L^\infty(\R^2))^{2\times 2}
 \cap (C^{0,\beta}_{\unif}(\R^2))^{2\times 2}
\]
Consequently, with $A=m\,a-B$, one has
\begin{equation}
 mLu=-\diverg(A\nabla u)
 \label{divformid}
\end{equation}
in the sense of distributions. Moreover
\[
 A=\Adiv_{\per}+A^{\e},
 \qquad A^{\e}\in (L^{M^*}(\R^2))^{2\times 2}\cap (L^\infty(\R^2))^{2\times 2}
 \cap (C^{0,\beta}_{\unif}(\R^2))^{2\times 2},
\]
and $A$ is coercive in the divergence-form sense. Moreover, \(A^{\e}(x),B^{\e}(x)\to0\)
uniformly as \(|x|\to\infty\). Here $\Adiv_{\per}=m^\per\aper-B^{\per}$.
\end{theorem}

\begin{remark}[Decay is not needed for Theorem~\ref{thm:divred}]
The tail condition
\[
 \lim_{R\to\infty}\sup_{|x|\ge R}|m^{\e}(x)|=0
\]
is not used in the divergence-form reduction. It is used only in the decaying uniqueness class of
Theorem~\ref{thm:inv} and Corollary~\ref{cor:unique}.
\end{remark}

\begin{remark}[Relation with the oscillatory problem]
Theorems~\ref{thm:central}--\ref{thm:divred} give the two-dimensional structural inputs used in the
advection--diffusion theory of Blanc--Le Bris--Lions. After Theorem~\ref{thm:divred}, the operator
has the invariant-measure divergence-form representative
\[
    mLu=-\diverg((ma-B)\nabla u).
\]
The boundary-value compactness argument and the local-profile interpretation are not reproved in
this paper. They are the divergence-form local-defect arguments of
\cite{BLLMilan,BLLCRAS,BLLLocalProfiles}, applied after the above reduction and after verifying the
corresponding hypotheses on the coefficient. Theorem~\ref{thm:divred} gives
\(A^{\e},B^{\e}\in L^{M^*}(\R^2)\cap L^\infty(\R^2)\cap
C^{0,\beta}_{\unif}(\R^2)\), and hence uniform decay at infinity. Any further hypotheses required
by a chosen divergence-form result must still be verified separately.
\end{remark}

\begin{remark}[Non-endpoint range]
The exponent \(q^*\) in Theorem~\ref{thm:central} is the planar analogue of the exponent appearing
in the scaling argument of Blanc--Le Bris--Lions. In the harmonic-coordinate variables the periodic
drift is removed, but the pull-back formula contains
\[
    D_x^2u=DP^T D_Y^2w\,DP+
    \sum_\alpha \partial_\alpha w\,D^2P_\alpha .
\]
Thus the term \(D^2P\,\nabla_Y w\) forces the homogeneous Sobolev exponent
\(1/q^*=1/q-1/2\). The lower endpoint is already false for the Laplacian: if
\(-\Delta u=f\) in \(\R^2\), \(f\in C_c^\infty\), and \(\int f\ne0\), then
\(|\nabla u(x)|\sim |x|^{-1}\) at infinity and \(\nabla u\notin (L^2(\R^2))^2\). The upper endpoint
\(q=2\) would require an \(L^\infty\) Calder\'on--Zygmund estimate. The restrictions
\(1<r,s<2\) keep the defect estimates away from the analogous Hodge endpoints.
\end{remark}

\section{Harmonic coordinates, Liouville theorem and central estimates}
\subsection{Periodic divergence form and harmonic coordinates}

This section contains the periodic objects and the coordinate gauge. We use only standard periodic Schauder--Fredholm theory, adjoint Fokker--Planck regularity and planar \(\sigma\)-harmonic mapping results.

\begin{lemma}[Periodic skew potential]\label{lem:per-skew}
Assume \(\mper\in C^{0,\theta_0}(\T^2)\) with \(\theta_0\in(0,\tau]\) is the periodic invariant density defined in Lemma~\ref{lem:per-density} and satisfies the periodic centering assumption~\ref{ass:centering}.
%There exists \(\theta_0\in(0,\tau]\) such that \(\mper\in C^{0,\theta_0}(\T^2)\). 
Define the distributional vector field
\begin{equation}\label{eq:Fper-def}
 F^{\per}:=\mper\bper+\diverg(\mper\aper).
\end{equation}
Then \(\diverg F^{\per}=0\) and \(\avg{F^{\per}}=0\). There is a unique zero-mean periodic stream function \(\psi^{\per}\in C^{0,\theta_0}(\T^2)\) such that
\begin{equation}\label{eq:div-Bper-Fper}
 \diverg \left(\psi^{\per}R\right)=F^{\per},
\end{equation}
where \[
 R=\begin{pmatrix}0&1\\-1&0\end{pmatrix}.
\]
Consequently
\begin{equation}\label{eq:Adivper-def}
 \Adiv_{\per}=\mper\aper-B^{\per}
\end{equation}
is periodic, H\"older continuous and coercive in the divergence-form sense.
Here $B^{\per}:=\psi^{\per}R$.
\end{lemma}

\begin{proof}
The adjoint equation for \(\mper\) gives \(\diverg F^{\per}=0\). Its mean is zero because
\(\avg{\mper\bper}=0\) and the mean of a periodic divergence vanishes. 
% The regularity and
% positivity of \(\mper\) follow from the local double-divergence theory of
% Bogachev--Shaposhnikov: higher integrability and H\"older regularity are contained in
% \cite[Theorems~2.1 and~3.1]{BogachevShaposhnikov}, while Harnack and strict positivity are in
% \cite[Corollary~3.6]{BogachevShaposhnikov}.

We define the zero-mean stream function on the torus by the Fourier formula
\[
 \widehat\psi^{\per}(k)=\frac{k_1\widehat F^{\per}_2(k)-k_2\widehat F^{\per}_1(k)}{i|k|^2},
 \qquad k\in2\pi\Z^2\setminus\{0\},\qquad \widehat\psi^{\per}(0)=0.
\]
The part of this multiplier acting on the non-divergence contribution to \(F^{\per}\) is of order
\(-1\), and the part acting on the term \(\diverg(\mper\aper)\) is a periodic zero-order
Calder\'on--Zygmund multiplier. The latter preserves H\"older continuity by the usual periodic
Calder\'on--Zygmund estimates; equivalently, one localizes the kernels and uses the homogeneous
singular-integral bounds in \cite[Ch.~5]{Grafakos}. Hence \(\psi^{\per}\in C^{0,\theta_0}\). The Fourier definition gives
\(\diverg(\psi^{\per}R)=F^{\per}\). The uniqueness of $\psi^{\per}$ follows from the zero-mean condition. 
Finally, since \(B^{\per}=\psi^{\per}R\) is skew-symmetric,
\[
 \xi\cdot\Adiv_{\per}(y)\xi=\mper(y)\,\xi\cdot\aper(y)\xi.
\]
The lower Harnack bound for \(\mper\) and the ellipticity of \(\aper\) give the coercivity of
\(\Adiv_{\per}\).
\end{proof}

\begin{lemma}[Periodic harmonic coordinates]\label{lem:harmcoord}
There are periodic functions \(\chi_\alpha\in C^{2,\tau'}_{\per}(\R^2)\) with $\alpha=1,\,2$, \(\avg{\chi_\alpha}=0\), such that
\begin{equation}\label{periodic-coord-cell}
 -\aper:D^2\chi_\alpha+\bper\cdot\nabla\chi_\alpha=-b_\alpha^{\per}
 \qquad\text{on }\T^2 .
\end{equation}
If \(Y:=P(x)=x+\chi(x)\), then \(L_{\per}P_\alpha=0\), \(\alpha=1,2\), and
\begin{equation}\label{sigmaharmonic}
 \diverg(\Adiv_{\per}\nabla P_\alpha)=0
\end{equation}
in distributions. Moreover \(P:\R^2\to\R^2\) is a uniformly bi-Lipschitz
\(C^{2,\tau'}\) diffeomorphism, \(D^2P_{\alpha}\in (L^\infty_{\per}(\R^2))^{2\times 2}\), and the derivatives of
\(\Theta=P^{-1}\) are periodic.
\end{lemma}

\begin{proof}
We first record the periodic invariant-measure identity, which is the formal reduction used in
\cite[Section~4]{BLL}. For every \(u\in C^2_{\loc}(\R^2)\) whose gradient is periodic, the identity
\[
 \mper L_{\per}u=-\diverg(\Adiv_{\per}\nabla u)
\]
holds in distributions. Indeed, expanding \(\Adiv_{\per}=\mper\aper-B^{\per}\) gives 
\(-\diverg(\Adiv_{\per}\nabla u)=-\partial_i(\mper a^{\per}_{ij})\partial_j u-\mper a^{\per}_{ij}\partial_{ij}u
+\partial_iB^{\per}_{ij}\partial_j u+B^{\per}_{ij}\partial_{ij}u\). The last term vanishes because
\(B^{\per}\) is skew-symmetric and \(D^2u\) is symmetric. Since
\(\partial_iB^{\per}_{ij}=\mper b^{\per}_j+\partial_i(\mper a^{\per}_{ij})\), we have $ \mper L_{\per}u=-\diverg(\Adiv_{\per}\nabla u)$. The computation is
unchanged for affine-plus-periodic functions with periodic gradients.

We now solve the cell problem. Fix \(0<\gamma\le\tau\) and consider
\[
  L_{\per}:C^{2,\gamma}_{\per}(\T^2)/\R\longrightarrow C^{0,\gamma}_{\per}(\T^2),
 \qquad  L_{\per}v=-\aper:D^2v+\bper\cdot\nabla v.
\]
By Schauder estimates and the Fredholm alternative on the compact torus
\cite[Ch.~5, Sec.~5.3 and Ch.~6, Secs.~6.1--6.3]{GT}, this operator is Fredholm of index zero.
Its kernel consists of constants, and its adjoint kernel is spanned by the normalized positive
periodic density \(\mper\). Hence the solvability condition for \eqref{periodic-coord-cell} is
\[
 \int_{\T^2}(-b_\alpha^{\per})\mper=0,
\]
which is exactly \(\avg{\mper\bper}=0\). The periodic Schauder estimate yields
\(\chi_\alpha\in C^{2,\tau'}_{\per}(\R^2)\). Since \(L_{\per}x_\alpha=b_\alpha^{\per}\), the map
\(P_\alpha(x)=x_\alpha+\chi_\alpha(x)\) satisfies \(L_{\per}P_\alpha=0\), and the preceding identity gives
\eqref{sigmaharmonic}.

The weak equation in \eqref{sigmaharmonic} places \(P\) in the affine-periodic
\(\sigma\)-harmonic class of Alessandrini--Nesi, with affine part \(I\) and coefficient
\(\sigma=\Adiv_{\per}\). Indeed \(P-I=\chi\in (C^{2,\tau'}_{\per}(\R^2))^2\subset (W^{1,2}_{\#}(\R^2))^2\), and
\(\Adiv_{\per}\) is periodic, bounded, H\"older continuous and uniformly elliptic in the coercive
sense. The periodic theorem of Alessandrini--Nesi \cite[Theorem~4.1]{AlessandriniNesi1} gives
that \(P\) is a global homeomorphism of \(\R^2\) and that the Jacobian has a positive sign almost
everywhere. Since \(\Adiv_{\per}\in (C^{0,\theta_0}(\R^2))^{2\times 2}\), their Lewy-type theorem
\cite[Theorem~1.1]{AlessandriniNesi2} gives \(\det DP(x)\ne0\) for every \(x\). Hence
\(\det DP>0\) everywhere.

Set
\[
 j_0=\min_{\T^2}\det DP>0,
 \qquad M_0=\|DP\|_{(L^\infty(\T^2))^{2\times 2}}<\infty .
\]
Since for a \(2\times2\) matrix \(F\), one has \(\sigma_{\min}(F)=|\det F|/\sigma_{\max}(F)\),
\[
 \frac{j_0}{M_0}|\xi|\le |DP(y)\xi|\le M_0|\xi|.
\]
The upper bound gives the Lipschitz estimate by integration along segments. The global
homeomorphism property and the uniform invertibility of \(DP\) give, by the inverse function
theorem, \(\Theta\in (C^{2,\tau'}_{\loc}(\R^2))^2\) and \(\|D\Theta\|_{(L^\infty(\mathbb{T}^2))^{2\times 2}}\le M_0/j_0\). The identity
\(P(x+k)=P(x)+k=Y+k\) implies \(\Theta(Y+k)=\Theta(Y)+k\), and differentiating gives the periodicity
of the derivatives of \(\Theta\). Finally, \(D^2P_{\alpha}=D^2\chi_{\alpha}\in (L^\infty_{\per}(\R^2))^{2\times 2}\) according to the same Schauder
estimate.
\end{proof}

\begin{proof}[Proof of Theorem~\ref{thm:gauge}]
Lemma~\ref{lem:harmcoord} gives the map \(P\). We prove the transformation formula. Let
\(Y=P(x)\), \(x=\Theta(Y)\), and set \(u(x)=w(Y)\). For any \(w\in C^2_{\rm loc}(\mathbb{R}^2)\), the chain rule gives
\[
 \partial_i u=(\partial_\alpha w)(P(x))\partial_iP_\alpha(x),
\]
and
\[
 \partial_{ij}u=(\partial_{\alpha\beta}w)(P(x))\partial_iP_\alpha(x)\partial_jP_\beta(x)
 +(\partial_\alpha w)(P(x))\partial_{ij}P_\alpha(x).
\]
Therefore
\[
 L_tu(x)=-(\haper+t\hae):D_Y^2w(Y)+\widehat b_t(Y)\cdot\nabla_Yw(Y),
\]
with
\[
 \haper(Y)=DP(x)\aper(x)DP(x)^T,
 \qquad
 \hae(Y)=DP(x)\arme(x)DP(x)^T,
\]
and
\[
 \widehat b_{t,\alpha}(Y)=L_tP_\alpha(x)
 =-a_t(x):D^2P_\alpha(x)+b_t(x)\cdot\nabla P_\alpha(x),\quad \alpha = 1,\,2.
\]
Since \(L_{\per}P_\alpha=0\), the periodic part of \(\widehat b_t\) vanishes and
\[
 \widehat b_{t,\alpha}(Y)=t\bigl[-\arme(x):D^2P_\alpha(x)+\brme(x)\cdot\nabla P_\alpha(x)\bigr]
 :=t\hbe_{\alpha}(Y),\quad \alpha = 1 ,\,2.
\]
This proves the displayed formula in the theorem and the definitions
\eqref{haperdef}--\eqref{hbedef}. If \(w\in W^{2,p}_{\loc}\), the same identities hold a.e.: mollify \(w\) on compact subsets of the range of \(P\), use the bi-Lipschitz change of variables and pass to the limit in \(L^p_{\loc}\). The distributional formulation follows by testing against compactly supported functions and the same change of variables. Finally, \(DP,D^2P\) are bounded and periodic, while \(\arme_{ij},\brme_i\in L^M(\R^2)\cap L^\infty(\R^2)\); hence \(\hae_{ij},\hbe_i\in L^M(\R^2)\cap L^\infty(\R^2)\).
\end{proof}

\subsection{Small critical-drift compactness}

The compactness lemma used in the blow-down argument is a strong-solution result. The statement
below spells out the sign convention and the ABP comparison used later.

\begin{lemma}[Strong small-$L^2$ drift compactness]\label{lem:smalldrift}
Let $U\Subset\R^2$ be bounded and let $R_n\to+\infty$. Let
\[
 A_n(x)=A_{\per}(R_nx)+E_n(x),
\]
where $A_{\per}\in (C^{0,\tau}_{\per}(\R^2))^{2\times 2}$ is symmetric and uniformly elliptic. Assume that it admits periodic
non-divergence correctors $\chi_M $ for every $M\in\mathcal S^2$, such that
\begin{equation}
 -A_{\per}(x):(M+D_y^2\chi_M(x))=-\overline A:M,
 \qquad \chi_M\in C^{2,\tau}_{\per},
 \label{cellgeneric}
\end{equation}
where the homogenized matrix $\overline A$ is uniformly elliptic. Assume
\[
 \|E_n\|_{(L^\infty(U))^{2\times 2}}\to0,
 \qquad \|C_n\|_{(L^2(U))^2}\to0,
\]
and, after discarding finitely many indices, $A_n$ are uniformly elliptic with constants independent
of $n$. Let
\[
 w_n\in W^{2,2}_{\loc}(U)\cap C(U),
 \qquad -A_n:D^2w_n+C_n\cdot\nabla w_n=0\quad\text{a.e. in }U,
\]
and assume $\|w_n\|_{L^\infty(U)}\leq M_0$. Then a subsequence of $w_n$ converges locally uniformly to a
continuous function $w$, and
\[
 -\overline A:D^2w=0
\]
in the viscosity sense in $U$.
\end{lemma}

\begin{proof}
We first obtain the compactness. Writing the equation as \(A_n:D^2w_n-C_n\cdot\nabla w_n=0\), we
are in the strong non-divergence setting of Mooney \cite[Definition~1.1]{Mooney}, with drift
\(-C_n\in (L^2(U))^2\). After scaling any ball \(B_{2r}(x_0)\Subset U\) to \(B_1\), the drift becomes
\(2rC_n(x_0+2rz)\), whose \(L^2\)-norm is \(\|C_n\|_{L^2(B_{2r}(x_0))}\). Mooney's H\"older
estimate \cite[Theorem~1.3]{Mooney} therefore gives uniform H\"older bounds on compact subsets
of \(U\), with constants depending only on the ellipticity constants, \(M_0\), and an upper bound
for \(\|C_n\|_{(L^2(U))^2}\). Arzel\`a--Ascoli yields, after extraction, local uniform convergence
\(w_n\to w\).

We shall use the following strong comparison consequence of Mooney's ABP estimate
\cite[Section~5, Theorem~5.1]{Mooney}. If \(Y\in W^{2,2}(B_\rho)\cap C(B_\rho)\), \(Y\le0\) on
\(\partial B_\rho\), and \(L_nY\le f\) a.e. in \(B_\rho\), where
\(L_n=-A_n:D^2+C_n\cdot\nabla\), then
\[
 \sup_{B_\rho}Y^+\le C_\rho\|f^+\|_{L^2(B_\rho)}.
\]
This is the strong \(W^{2,2}\) comparison estimate; no endpoint viscosity ABP statement is used.
The sign convention follows by applying Mooney's estimate to \(-Y\) for the operator
\(A_n:D^2-C_n\cdot\nabla\).

We prove the viscosity subsolution inequality. Let \(\phi\in C^2(U)\) touch \(w\) strictly from
above at \(x_0\). Choose \(B_\rho(x_0)\Subset U\) and \(\kappa>0\) so that
\(w(x_0)=\phi(x_0)\) and \(w-\phi\le -4\kappa\) on \(\partial B_\rho(x_0)\). We claim \(-\overline A:D^2\phi(x_0)\le0\). Let
\(M=D^2\phi(x_0)\). Assume, for contradiction, that \(-\overline A:D^2\phi(x_0)=-\overline A:M>0\). With the corrector from
\eqref{cellgeneric}, set
\[
 \phi_n(x)=\phi(x)+R_n^{-2}\chi_M(R_nx).
\]
Then \(\phi_n\to\phi\) uniformly and \(\nabla\phi_n,D^2\phi_n\) are uniformly bounded. The cell
problem, the smallness of \(E_n\), and the continuity of \(D^2\phi\) in a sufficiently small ball give,
for large \(n\),
\[
 L_n\phi_n\ge \frac\vartheta2+h_n,
 \qquad h_n=C_n\cdot\nabla\phi_n,
 \qquad \|h_n\|_{L^2(B_\rho)}\to0.
\]
For \(Y_n=w_n-\phi_n+2\kappa\), local uniform convergence gives \(Y_n\le0\) on
\(\partial B_\rho\) for large \(n\). Since $\lim\limits_{n}Y_n(x_0)=2\kappa$, then \(\sup\limits_{B_\rho} Y_n^+\ge\kappa\) holds for large \(n\). Since \(L_nw_n=0\), we have
\(L_nY_n\le-\vartheta/2-h_n\); the positive part of the right-hand side is bounded by \(|h_n|\). The
comparison estimate gives \(\kappa\le C\|h_n\|_{L^2(B_\rho)}\to0\), a contradiction. Hence
\(-\overline A:D^2\phi(x_0)\le0\). Non-strict contacts are obtained by adding
\(\varepsilon|x-x_0|^2\) and letting \(\varepsilon\downarrow0\).

The supersolution inequality is identical. If \(\phi\) touches \(w\) strictly from below and
\(-\overline A:D^2\phi(x_0)<0\), the same perturbed test function satisfies the reversed inequality;
applying the comparison estimate to \(\phi_n-w_n+2\kappa\) gives the same contradiction. Thus
\(w\) is a viscosity solution of \(-\overline A:D^2w=0\).
\end{proof}
\begin{remark}[Periodic non-divergence correctors]\normalfont
The corrector assumption in \eqref{cellgeneric} is standard. Let \(\rho_{\per}\)
be the normalized positive adjoint density for \(-A_{\per}:D^2\), that is
\(\partial_{ij}(A_{\per}^{ij}\rho_{\per})=0\) on \(\T^2\) and
\(\langle \rho_{\per}\rangle=1\). For \(M\in\mathcal S^2\), set
\[
    \overline A:M
    =
    \int_{\T^2}\rho_{\per}(y)\,A_{\per}(y):M\,dy .
\]
Then \(A_{\per}:M-\overline A:M\) is orthogonal to the adjoint kernel, and the
Fredholm alternative on the torus gives a unique zero-mean periodic solution of
\[
    -A_{\per}:D^2\chi_M=A_{\per}:M-\overline A:M .
\]
Equivalently, \(-A_{\per}(y):(M+D_y^2\chi_M)=-\overline A:M\). Since
\(A_{\per}\in (C_{\per}^{0,\tau}(\R^2))^{2\times 2}\), the periodic Schauder estimate gives
\(\chi_M\in C_{\per}^{2,\tau'}(\R^2)\) for some \(\tau'\in(0,\tau]\); see
Avellaneda--Lin \cite[p.~152, (3.1)--(3.4), and p.~170, (4.10)--(4.11)]{AvellanedaLinND}
and Gilbarg--Trudinger \cite[Ch.~6, Theorems~6.2 and~6.6]{GT}. Finally,
\(\overline A\) is uniformly elliptic by the positivity of \(\rho_{\per}\) and the
uniform ellipticity of \(A_{\per}\).
\end{remark}
\subsection{The finite-energy Liouville theorem}

\begin{lemma}[Good scales]\label{lem:goodscales}
Let $N:[1,\infty)\to(0,\infty)$ be nondecreasing and satisfy $N(R)\leq CR^\alpha$ for some
$\alpha<\beta<1$. Then there exist $R_n\to\infty$ such that
\begin{equation}
 N(SR_n)\leq 2S^\beta N(R_n)
 \quad\text{for every }S\geq1.
 \label{goodscale}
\end{equation}
\end{lemma}

\begin{proof}
Suppose, for contradiction, that the conclusion is false. Then there exists $R_0\geq1$ such
that every $R\geq R_0$ is a bad radius: for each such $R$ one can choose $S\geq1$ satisfying
\[
 N(SR)>2S^\beta N(R).
\]
Starting from $R_0$, choose inductively $S_k\geq1$ such that
\[
 N(S_kR_k)>2S_k^\beta N(R_k),
 \qquad R_{k+1}:=S_kR_k.
\]
Iteration gives
\[
 N(R_k)>2^k\left(\frac{R_k}{R_0}\right)^\beta N(R_0).
\]
Combining this with $N(R_k)\leq CR_k^\alpha$ yields
\[
 2^kN(R_0)<CR_0^\beta R_k^{\alpha-\beta}\leq CR_0^\alpha,
\]
because $R_k\geq R_0$ and $\alpha-\beta<0$. This is impossible as $k\to\infty$.
\end{proof}

\begin{proof}[Proof of Theorem~\ref{thm:liouville}]
Assume that $u$ is not constant and set $M(R)=\osc_{B_R}u$. Choose $R_*\geq1$ such that
$M(R_*)>0$, and apply Lemma~\ref{lem:goodscales} to the positive nondecreasing function
$N(R):=M(R_*R)$. Indeed, Morrey's estimate
\cite[Ch.~7, Theorem~7.17, formula~(7.42)]{GT} gives
$N(R)\le C(R_*R)^{1-2/p}\|\nabla u\|_{(L^p(\R^2))^2}$. Choose
$\beta\in(1-2/p,1)$ and relabel the resulting radii $R_*R_n$ as $R_n$. Then
\begin{equation}\label{osc-growth}
    M(SR_n)\le 2S^\beta M(R_n),\qquad S\ge1.
\end{equation}
We normalize by
\[
    v_n(x)=\frac{u(R_nx)-u(0)}{M(R_n)} .
\]
Then $v_n(0)=0$, $\osc_{B_1}v_n=1$, and \eqref{osc-growth} gives
$\osc_{B_S}v_n\le2S^\beta$ for every $S\ge1$.

We now pass to harmonic coordinates. Put $z:=\Phi_n(x)=R_n^{-1}P(R_nx)$ and write
$v_n(x)=w_n(z)$. Since $P(x)-x$ is bounded and $P$ is uniformly bi-Lipschitz, one has $\Phi_n(x)=x+\frac{\chi(R_nx)}{R_n}$ and
$\Phi_n\to I$ locally uniformly, and fixed annuli are mapped into comparable annuli. If
$u(x)=\widehat u(P(x))$ with $Y=P(x)$, Theorem~\ref{thm:gauge} gives
$-(\haper+t\hae):D_Y^2\widehat u+t\hbe\cdot\nabla_Y\widehat u=0$. Since $w_n(z)=\frac{\widehat u(R_nz)-u(0)}{M(R_n)}$, on every
annulus $U\Subset\R^2\setminus\{0\}$,
\begin{equation}\label{rescaled-liouville-eq}
    -(\haper(R_nz)+E_n(z)):D_z^2w_n+C_n(z)\cdot\nabla_zw_n=0,
    \qquad
    E_n(z)=t\hae(R_nz),\qquad C_n(z)=tR_n\hbe(R_nz),
\end{equation}
where $z=R_n^{-1}P(R_nx)$.
The coefficient error vanishes uniformly on $U$. Indeed, any uniformly H\"older function in
$L^q(\R^2)$, $q<\infty$, tends uniformly to zero at infinity; otherwise one extracts disjoint balls on
which the function is bounded below. Applying this to the entries of $\arme$ and using the bounded
periodic factors in \eqref{haedef} gives $\|E_n\|_{L^\infty(U)}\to0$. For the drift, the change of
variables $y:=R_nx=\Theta(R_nz)\in \Theta(R_n U)$ yields
\[
    \|C_n\|_{(L^2(U))^2}^2
    \le C\int_{\Theta(R_nU)}(|\arme(y)|^2+|\brme(y)|^2)\,dy\to0,
\]
because $\Theta(Y)=Y+O(1)$, the sets $\Theta(R_nU)$ leave every compact subset, and
$\arme\in (L^M(\R^2)\cap L^\infty(\R^2))^{2\times 2}$, $\brme\in (L^M(\R^2)\cap L^\infty(\R^2))^{2}$ and both of them vanish at infinity.

The periodic coefficient $\haper$ has the usual non-divergence correctors by the periodic Schauder
Fredholm theory, and the corresponding homogenized matrix $A_{\hom}$ is uniformly elliptic. Thus
Lemma~\ref{lem:smalldrift} applies on annuli. After a diagonal extraction,
\[
    w_n\to w\quad\text{locally uniformly in }\R^2\setminus\{0\},
    \qquad -A_{\hom}:D^2w=0
\]
in the viscosity sense. The same annular convergence holds for $v_n$, because $\Phi_n\to I$. The
normalization of $v_n$ gives boundedness near the origin; hence the isolated singularity of $w$ is
removable for constant-coefficient uniformly elliptic equations, by Gilbarg--Serrin
\cite[Section~5, Theorem~6]{GilbargSerrin}. Thus $-A_{\hom}:D^2w=0$ in all of $\R^2$.
Passing \eqref{osc-growth} to the limit gives $\osc_{B_S}w\le2S^\beta$. The constant-coefficient
interior gradient estimate, after a linear change of variables,
\cite[Ch.~6, Lemma~6.1(a)]{GT}, gives $\|\nabla w\|_{L^\infty(B_{S/2})}\le CS^{-1}\osc_{B_S}w
\le CS^{\beta-1}$. Since $\beta<1$, letting $S\to\infty$ gives $w\equiv c$.

It remains to push the annular convergence through the origin. Fix $0<\delta<1/4$ and set
$U_{n,\delta}=\Phi_n(B_\delta)$. For large $n$ one has
$U_{n,\delta}\subset B_{2\delta}$ and
$\partial U_{n,\delta}\subset B_{2\delta}\setminus B_{\delta/2}$; hence the annular convergence gives
$\sup_{\partial U_{n,\delta}}|w_n-c|\to0$. On $U_{n,\delta}$ the functions $w_n-c$ and
$-(w_n-c)$ solve the strong equation \eqref{rescaled-liouville-eq}. The coefficients need not be
small there, since the domain may contain the blown-down defect core, but for each fixed $n$ the
strong ABP comparison following from Mooney's estimate \cite[Section~5, Theorem~5.1]{Mooney}
applies. Although stated above on balls, the comparison extends to bounded $C^2$ domains by
exhaustion; the right-hand side is zero, so no uniform geometric constant is used. Thus
\[
    \sup_{U_{n,\delta}}|w_n-c|\le \sup_{\partial U_{n,\delta}}|w_n-c|\to0.
\]
Since \(v_n=w_n\circ\Phi_n\) and \(U_{n,\delta}=\Phi_n(B_\delta)\), the comparison estimate on
\(U_{n,\delta}\) gives
\[
    \sup_{B_\delta}|v_n-c|
    \le
    \sup_{U_{n,\delta}}|w_n-c|
    \to0.
\]
On the complement, we use that \(\Phi_n\to I\) locally uniformly and that the maps
\(\Phi_n\) are uniformly bi-Lipschitz. Thus, for \(n\) large,
\(\Phi_n(\overline{B_1\setminus B_\delta})\subset
\overline{B_2\setminus B_{\delta/2}}\). Since \(w_n\to c\) locally uniformly in
\(\mathbb R^2\setminus\{0\}\), it follows that
\[
    \sup_{B_1\setminus B_\delta}|v_n-c|
    \le
    \sup_{\overline{B_2\setminus B_{\delta/2}}}|w_n-c|
    \to0.
\]
Hence \(v_n\to c\) uniformly in \(B_1\), contradicting
\(\operatorname{osc}_{B_1}v_n=1\). This proves the Liouville theorem.
\end{proof}

\subsection{The central estimate}

\begin{lemma}[Periodic whole-space starting estimate]
\label{lem:perwhole}
Let $1<q<2$ and let $q^*$ be defined by
\begin{equation}\label{qstar-def-perwhole}
    \frac1{q^*}=\frac1q-\frac12 .
\end{equation}
For every
    $f\in  L^q(\R^2)\cap L^{q^*}(\R^2)$,
there exists a solution, unique modulo constants, of
\begin{equation}\label{perwhole-equation}
    L_{\rm per}u=f
    \qquad\text{in }\R^2,
\end{equation}
such that
\[
    \nabla u\in(L^{q^*}(\R^2))^2,\quad D^2u\in (L^{q^*}(\R^2))^{2\times 2},
\]
and
\begin{equation}\label{perwhole-est}
    \|\nabla u\|_{(L^{q^*}(\R^2))^2}
    +\|D^2u\|_{(L^{q^*}(\R^2))^{2\times 2}}
    \le
    C_q\bigl\|f\|_{(L^q\cap L^{q^*})(\R^2)}.
\end{equation}
Moreover
\begin{equation}\label{perwhole-divid}
    -\diverg(\Adiv_{\per}\nabla u)=\mper f
    \qquad\text{in }\mathcal D'(\R^2).
\end{equation}
\end{lemma}

\begin{proof}
Let \(Y=P(x)\), \(x=\Theta(Y)=P^{-1}(Y)\), and write \(u(x)=w(Y)\). The chain rule gives
\(\nabla_xu=(DP)^T\nabla_Y\,w\) and
\(D_x^2u=(DP)^T\,D_Y^2\,wDP+\sum_\alpha \partial_{Y_\alpha} w\,D^2P_\alpha\). Since
\(L_{\per}P_\alpha=0\), the first-order term disappears and the equation \eqref{perwhole-equation} becomes
\[
 -\haper:D_Y^2w=\widehat f,
 \qquad \widehat f(Y)=f(\Theta(Y)).
\]
The map \(P\) is uniformly bi-Lipschitz with Jacobian bounded above and below; hence
\(\|\widehat f\|_{L^p}(\R^2)\le C\|f\|_{L^p}(\R^2)\) for every \(1\le p\le\infty\). The coefficient \(\haper\) is
periodic, H\"older continuous, symmetric and uniformly elliptic.

For smooth compactly supported \(f\), the scalar non-divergence whole-space estimate of
Avellaneda--Lin \cite[Corollary~3]{AvellanedaLinLp}, based on Theorem~B therein, gives a solution
of the transformed equation, unique modulo affine functions, such that
\[
 \|D_Y^2w\|_{(L^q(\R^2))^{2\times 2}}+\|D_Y^2w\|_{(L^{q^*}(\R^2))^{2\times 2}}
 \le C\|f\|_{(L^q\cap L^{q^*})(\R^2)}.
\]
Choosing the affine representative in the homogeneous Sobolev class and using
\(1/q^*=1/q-1/2\), the homogeneous Sobolev inequality gives
\(\|\nabla_Yw\|_{L^{q^*}}\le C\|D_Y^2w\|_{L^q}\). Pulling the estimate back to \(x\)-variables and
using the displayed chain-rule identity, the boundedness of \(DP,D^2P\), and the bi-Lipschitz
change of variables, yields \eqref{perwhole-est}.

For general \(f\in L^q(\R^2)\cap L^{q^*}(\R^2)\), choose \(f_n\in C_c^\infty\) converging to \(f\) in this
space. The corresponding normalized solutions form a Cauchy sequence in the norm
\(\|\nabla\cdot\|_{(L^{q^*}(\R^2))^{2}}+\|D^2\cdot\|_{(L^{q^*}(\R^2))^{2\times 2}}\). After fixing
the normalization \(|B_1|^{-1}\int_{B_1}u_n=0\), local Poincare inequalities show that the
sequence converges in \(W^{2,q^*}_{\rm loc}\), and its limit gives the desired solution.
Uniqueness modulo constants follows from Theorem~\ref{thm:liouville} with $t=0$. Finally, the
identity \eqref{perwhole-divid} follows for smooth data from the periodic invariant-measure
identity and then by approximation.
\end{proof}

\begin{proof}[Proof of Theorem~\ref{thm:central}]
We employ a continuation argument based on Proposition~2.1 of Blanc, Le Bris and Lions
\cite{BLL}, as corrected in \cite{BLLErratum}. As in their proof, our argument is divided into
three steps. The first two steps follow the corrected proposition almost verbatim. The third step,
however, requires a fundamentally different approach, since the original argument does not extend
to the two-dimensional setting due to the lack of the embedding inequalities (i.e. the proof of
(29) in \cite{BLL}) on which it relies.

Fix $1<q<2$, and let $q^*$ be defined by
\[
 \frac1{q^*}=\frac1q-\frac12.
\]
For $t\in[0,1]$, recall that
\[
 a_t=\aper+t \arme,
 \qquad b_t=\bper+t \brme,
 \qquad L_tu=-a_t:D^2u+b_t\cdot\nabla u.
\]
Set \(p=q^*\), let \(\mathcal Y_q=L^q(\R^2)\cap L^p(\R^2)\), with the sum norm, and introduce the fixed graph space
\[
 \mathcal X_q=\left\{u\in W^{2,p}_{\rm loc}(\R^2)/\R:
 \nabla u,D^2u\in L^p(\R^2),\ L_{\per}u\in\mathcal Y_q\right\},
\]
with norm
\[
 \|u\|_{\mathcal X_q}=\|\nabla u\|_{L^p}+\|D^2u\|_{L^p}
 +\|L_{\per}u\|_{\mathcal Y_q}.
\]
Both \(\mathcal X_q\) and \(\mathcal Y_q\) are Banach spaces. For \(\mathcal X_q\), this follows by
fixing \(|B_1|^{-1}\int_{B_1}u=0\), applying local Poincare inequalities to a Cauchy sequence,
and identifying the distributional limit of \(L_{\per}u_n\). Moreover
\(K=-\arme:D^2+\brme\cdot\nabla\) is bounded from \(\mathcal X_q\) to \(\mathcal Y_q\), and
\(L_t=L_{\per}+tK\).

We denote by $\mathcal P(t)$ the assertion that \(L_t:\mathcal X_q\to\mathcal Y_q\) is bijective.
Whenever \(\mathcal P(t)\) holds, set
\[
 C(t):=\sup_{u\in\mathcal X_q\setminus\{0\}}
 \frac{\|\nabla u\|_{L^p(\R^2)}+\|D^2u\|_{L^p(\R^2)}}
      {\|L_tu\|_{\mathcal Y_q}}<\infty.
\]
Set
\[
 I=\{t\in[0,1]:\mathcal P(s)\text{ holds for every } s\in[0,t]\}.
\]
We prove that $I$ is not empty, open, and closed.

\emph{Step 1: $0\in I$.} Let $f\in L^q(\R^2)\cap L^{q^*}(\R^2)$.  Lemma~\ref{lem:perwhole} gives a solution of
$L_{\per}u=f$, unique modulo constants, and the estimate
\[
 \|\nabla u\|_{(L^{q^*}(\R^2))^2}+\|D^2u\|_{(L^{q^*}(\R^2))^{2\times 2}}
 \leq C_q\|f\|_{(L^q\cap L^{q^*})(\R^2)}.
\]
% For general $f\in L^q\cap L^{q^*}$, take $f_n\in C_c^\infty(\R^2)$ converging to $f$ in both
% norms. Lemma~\ref{lem:perwhole} gives solutions $u_n$, unique modulo constants, and the
% sequence of pairs $(\nabla u_n,D^2u_n)$ is Cauchy in
% $L^{q^*}(\R^2)\times L^{q^*}(\R^2)$. After subtracting constants on bounded balls and using
% Poincare, $u_n$ converges locally in distributions to a function $u$ with
% \[
%  \nabla u,D^2u\in L^{q^*}(\R^2),\qquad L_{\per}u=f.
% \]
% The estimate passes to the limit, and uniqueness modulo constants follows from
% Theorem~\ref{thm:liouville} with $t=0$. 
Therefore $0\in I$.

\emph{Step 2: $I$ is open.} Assume that $\mathcal P(t)$ holds. For $|\tau|$ small, solve
\[
 L_{t+\tau}u=f
\]
by the iteration
\[
 L_tu_{k+1}=f+\tau\bigl(\arme:D^2u_k-\brme\cdot\nabla u_k\bigr),
 \qquad u_0=0.
\]
It remains to verify that the right-hand side belongs to $L^q(\R^2)\cap L^{q^*}(\R^2)$. Since
$a_{ij}^{\rm e},b_i^{\rm e}\in L^M(\R^2)\cap L^\infty(\R^2)$ with $M<2$, we have $a_{ij}^{\rm e},b_i^{\rm e}\in L^2(\R^2)\cap L^\infty(\R^2)$. Hence, using
\[
 \frac1q=\frac12+\frac1{q^*},
\]
H\"older's inequality gives
\[
 \|\arme:D^2u_k\|_{L^q(\R^2)}\leq \|\arme\|_{(L^2(\R^2))^{2\times 2}}\|D^2u_k\|_{(L^{q^*}(\R^2))^{2\times 2}},
\]
and similarly
\[
 \|\brme\cdot\nabla u_k\|_{L^q(\R^2)}\leq \|\brme\|_{(L^2(\R^2))^{2}}\|\nabla u_k\|_{(L^{q^*}(\R^2))^{2}}.
\]
The $L^{q^*}$ part is controlled by the $L^\infty$ norms:
\[
 \|\arme:D^2u_k\|_{L^{q^*}(\R^2)}
 \leq \|\arme\|_{(L^\infty(\R^2))^{2\times 2}}\|D^2u_k\|_{(L^{q^*}(\R^2))^{2\times 2}},
\]
and similarly for $\brme\cdot\nabla u_k$. Thus the perturbation operator
\[
 u\mapsto \arme:D^2u-\brme\cdot\nabla u
\]
is bounded from the solution norm
\[
 \|\nabla u\|_{(L^{q^*}(\R^2))^2}+\|D^2u\|_{(L^{q^*}(\R^2))^{2\times 2}}
\]
to $L^q\cap L^{q^*}$. Since $\mathcal P(t)$ holds, the above iteration is well defined. For
$|\tau|$ small enough, the series $\sum_{k\ge0}(u_{k+1}-u_k)$ converges geometrically in the
solution norm, and hence in \(\mathcal X_q\) by the defining equations and the boundedness of
\(K\). Its limit solves $L_{t+\tau}u=f$, satisfies \eqref{centralestimate}, and uniqueness
modulo constants follows by applying the same contraction argument to the homogeneous equation.
Therefore $I$ is open.

\emph{Step 3: $I$ is closed.} Let $t_n\in I$ and $t_n\to t$, and set
\(C_n:=C(t_n)\). If $C:=\sup_n C_n<\infty$, then for fixed
$f\in L^q(\R^2)\cap L^{q^*}(\R^2)$ the solutions $u_n$ of
\[
 L_{t_n}u_n=f
\]
satisfy
\[
 \|\nabla u_n\|_{(L^{q^*}(\R^2))^2}+\|D^2u_n\|_{(L^{q^*}(\R^2))^{2\times 2}}
 \leq C\|f\|_{(L^q\cap L^{q^*})(\R^2)}.
\]
After subtracting constants, normalize by $|B_1|^{-1}\int_{B_1}u_n=0$. The sequences of elements in
$\nabla u_n$ and $D^2u_n$ are bounded in the global space $L^{q^*}(\R^2)$; after extracting a
subsequence they converge weakly in $L^{q^*}(\R^2)$. Local
Poincare inequality and interior compactness identify these weak limits with $\nabla u$ and $D^2u$ for a
function $u\in W^{2,q^*}_{\loc}(\R^2)$. Since
\[
 L_tu_n=f+(t_n-t)\bigl(\arme:D^2u_n-\brme\cdot\nabla u_n\bigr),
\]
and the last term tends to zero in $L^q(\R^2)\cap L^{q^*}(\R^2)$ by the same estimates used in the openness
step, the weak limit solves $L_tu=f$. Since $L_{\per}u=f-tKu\in\mathcal Y_q$, we have
$u\in\mathcal X_q$. Thus $L_t$ is surjective on the graph space, while
Theorem~\ref{thm:liouville} gives injectivity. The estimate passes to the limit by weak lower
semicontinuity in the global $L^{q^*}$ spaces. Hence $\mathcal P(t)$ holds if the constants
$C_n$ remain bounded. Moreover, for every $s<t$, one has $s<t_n$ for all sufficiently large $n$,
and hence $\mathcal P(s)$ because $t_n\in I$. Thus $t\in I$ whenever $\sup_nC_n<\infty$.

It remains to show that the constants $C_n$ cannot blow up. Arguing by contradiction as in
\cite[Proposition~2.1]{BLL}, with the corrections in \cite{BLLErratum}, suppose that $C_n\to\infty$. By the definition of the best constants,
after normalization there exist $u_n$ and $f_n$ such that
\[
 L_{t_n}u_n=f_n,
 \qquad \|f_n\|_{(L^q\cap L^{q^*})(\R^2)}\to0,
\]
and
\begin{equation}
 \|\nabla u_n\|_{(L^{q^*}(\R^2))^2}+\|D^2u_n\|_{(L^{q^*}(\R^2))^{2\times 2}}=1.
 \label{normalization}
\end{equation}
Rewriting the equation with the limiting operator gives
\[
 L_tu_n=f_n+(t_n-t)\bigl(\arme:D^2u_n-\brme\cdot\nabla u_n\bigr).
\]
Again by the perturbation estimates above, the last term tends to zero in $L^q\cap L^{q^*}(\R^2)$.
Changing the notation of $f_n$, we may therefore assume
\begin{equation}
 L_tu_n=f_n,
 \qquad \|f_n\|_{(L^q\cap L^{q^*})(\R^2)}\to0.
 \label{limitoperatorun}
\end{equation}

We claim that there exist $R>0$ and $\eta>0$ such that, up to a subsequence,
\begin{equation}
 \|\nabla u_n\|_{(L^{q^*}(B_R))^{2}}+\|D^2u_n\|_{(L^{q^*}(B_R))^{2\times2}}\geq\eta.
 \label{nonvanishingclaim}
\end{equation}
Assume the contrary. Then for every $R>0$,
\[
 \|\nabla u_n\|_{(L^{q^*}(B_R))^{2}}+\|D^2u_n\|_{(L^{q^*}(B_R))^{2\times2}}\xrightarrow{n\to\infty}0.
\]
We next show that
\begin{equation}
 \arme:D^2u_n-\brme\cdot\nabla u_n\xrightarrow{n\to\infty}0
 \quad\text{in }L^q(\R^2)\cap L^{q^*}(\R^2).
 \label{defecttermvanishes}
\end{equation}
For the $L^q$ norm, split $\R^2=B_R\cup B_R^c$. In $B_R$, due to the contrary assumption and $a_{ij}^{\rm e},b_i^{\rm e}\in L^M(\R^2)\cap L^\infty(\R^2)$, the local norm above tends to zero.
On $B_R^c$,
\[
 \|\arme:D^2u_n\|_{L^q(B_R^c)}
 \leq \|\arme\|_{(L^2(B_R^c))^{2\times 2}}\|D^2u_n\|_{(L^{q^*}(B_R^c))^{2\times 2}},
\]
Since $\sup_n\|D^2u_n\|_{(L^{q^*}(\R^2))^{2\times 2}}\le 1$ by \eqref{normalization} and
$a_{ij}^{\rm e}\in L^M(\R^2)\cap L^\infty(\R^2)\subset L^2(\R^2)$, this gives
\[
 \lim_{R\to\infty}\sup_n
 \|\arme:D^2u_n\|_{L^q(B_R^c)}=0.
\]
The same argument gives
\[
 \lim_{R\to\infty}\sup_n
 \|\brme\cdot\nabla u_n\|_{L^q(B_R^c)}=0.
\]

For the $L^{q^*}$ norm, due to the contrary assumption and $a_{ij}^{\rm e},b_i^{\rm e}\in L^M(\R^2)\cap L^\infty(\R^2)$, the local part again tends to zero. In $B_R^c$,
\[
 \|\arme:D^2u_n\|_{L^{q^*}(B_R^c)}
 \leq \left(\max\limits_{ij}\sup\limits_{|x|>R}|\arme_{ij}(x)|\right)\,\|D^2u_n\|_{(L^{q^*}(B_R^c))^{2\times 2}}.
\]
The factor $\max\limits_{ij}\sup\limits_{|x|>R}|\arme_{ij}(x)|$ tends to zero as $R\to\infty$. Indeed, if this limit were not
zero, then there would exist $\eps>0$ and points $x_k\to\infty$ such that
$|a_{ij}^{\rm e}(x_k)|\ge\eps$. Uniform H\"older continuity of $\arme$ would give a radius $r>0$, independent of $k$,
such that $|a_{ij}^{\rm e}|\ge\eps/2$ in $B_r(x_k)$. Passing to a disjoint subsequence of these balls would
contradict $a_{ij}^{\rm e}\in L^M(\R^2)$. The same argument applies
to $\brme$. Hence \eqref{defecttermvanishes} follows.

Now
\[
 L_{\per}u_n=L_tu_n+t\bigl(\arme:D^2u_n-\brme\cdot\nabla u_n\bigr)\xrightarrow{n\to\infty}0
 \quad\text{in }L^q\cap L^{q^*}(\R^2).
\]
The periodic estimate from Step~1 then gives
\[
 \|\nabla u_n\|_{(L^{q^*}(\R^2))^{2}}
 +\|D^2u_n\|_{(L^{q^*}(\R^2))^{2\times2}}
 \le C\|L_{\per}u_n\|_{(L^q\cap L^{q^*})(\R^2)}\to0,
\]
which contradicts normalization \eqref{normalization}. Therefore the claim
\eqref{nonvanishingclaim} holds.

Fix $R>0$ and $\eta>0$ given by the claim. Subtract constants so that, for instance,
\[
 \frac1{|B_1|}\int_{B_1}u_n=0.
\]
By Poincare inequality on bounded balls and \eqref{normalization}, the sequence is bounded in
$W^{2,q^*}(B_S)$ for every fixed $S>0$. Since $q^*>2$, after a diagonal extraction there exists
$u\in W^{2,q^*}_{\loc}(\R^2)$ such that
\[
 u_n\rightharpoonup u\quad\text{in }W^{2,q^*}_{\loc}(\R^2),
 \qquad
 u_n\to u\quad\text{in }W^{1,q^*}_{\loc}(\R^2).
 %C^1_{\loc}(\R^2).
\]
At the same time, the global normalization gives boundedness of
$(\nabla u_n,D^2u_n)$ in
$(L^{q^*}(\R^2))^2\times (L^{q^*}(\R^2))^{2\times 2}$. Passing to a further subsequence, if necessary,
\[
 \nabla u_n\rightharpoonup U \quad\text{weakly in }(L^{q^*}(\R^2))^2,
 \qquad D^2u_n\rightharpoonup V
 \quad\text{weakly in }(L^{q^*}(\R^2))^{2\times 2}.
\]
The local convergence identifies $U=\nabla u$ and $V=D^2u$ distributionally. Hence
$\nabla u,D^2u\in L^{q^*}(\R^2)$, with the corresponding weak lower-semicontinuity bound.
Passing to the limit in \eqref{limitoperatorun} gives
\[
 L_tu=0\quad\text{in }\R^2.
\]

We next prove strong convergence on the ball from \eqref{nonvanishingclaim}. Set
$v_n=u_n-u$. Since $L_tu_n=f_n$ and $L_tu=0$, we have
\[
 L_tv_n=f_n\quad\text{in }B_{2R}.
\]
After setting $\Omega'=B_R$, $\Omega=B_{2R}$, $p=q^*$, due to $a_t\in C^{0,\tau}_{\unif}(\R^2)$ and $b_t\in L^\infty(B_{2R})$, it follows from Gilbarg--Trudinger \cite[Theorem 9.11]{GT} that 
\[
 \|v_n\|_{W^{2,q^*}(B_R)}
 \leq C\bigl(\|f_n\|_{L^{q^*}(B_{2R})}+\|v_n\|_{L^{q^*}(B_{2R})}\bigr).
\]

% We use Gilbarg--Trudinger \cite[Theorem 9.11]{GT} in the following form. Let
% $\Omega'\Subset\Omega$, $1<p<\infty$, and
% \[
%  \mathcal Lv=A^{ij}(x)D_{ij}v+\beta^i(x)D_iv
% \]
% with $A$ uniformly elliptic, $A$ uniformly continuous in $\Omega$, and
% $\beta\in L^\infty(\Omega)$. Then every $v\in W^{2,p}(\Omega)$ satisfies
% \[
%  \|v\|_{W^{2,p}(\Omega')}
%  \leq C\bigl(\|\mathcal Lv\|_{L^p(\Omega)}+\|v\|_{L^p(\Omega)}\bigr),
% \]
% where $C$ depends on $p$, $\Omega'$, $\Omega$, the ellipticity constants, the modulus of
% continuity of $A$, and $\|\beta\|_{L^\infty(\Omega)}$, but not on $v$.

% Apply this theorem with $\Omega'=B_R$, $\Omega=B_{2R}$, $p=q^*$, and with
% \[
%  \mathcal L_tv=a_t^{ij}D_{ij}v-b_t^iD_iv=-L_tv.
% \]
% The hypotheses are satisfied because $a_t=\aper+t \arme$ is uniformly elliptic and belongs to
% $C^{0,\tau}_{\unif}$, while $b_t=\bper+t \brme\in L^\infty$. Therefore

The first term tends to zero because $f_n\to0$ in $L^{q^*}(\R^2)$. The second term tends to zero as follows. The local $W^{2,q^*}$ bounds obtained above, together with the normalization by constants and $q^*>2$, give after extraction strong convergence in $L^{q^*}_{\loc}(\R^2)$ by Rellich compactness, in fact locally in $C^1$ by Morrey's embedding. Thus
\[
 \|v_n\|_{W^{2,q^*}(B_R)}\to0,
\]
and consequently
\[
 u_n\to u\quad\text{strongly in }W^{2,q^*}(B_R).
\]
The non-vanishing lower bound \eqref{nonvanishingclaim} passes to the limit:
\[
 \|\nabla u\|_{(L^{q^*}(B_R))^2}+\|D^2u\|_{(L^{q^*}(B_R))^{2\times 2}}\geq\eta.
\]
Hence $u$ is not constant. This contradicts Theorem~\ref{thm:liouville}, since $q^*>2$. Thus
$\sup_n C_n<\infty$, $I$ is closed, and hence $I=[0,1]$. Finally,
\(\sup_{t\in[0,1]}C(t)<\infty\): otherwise compactness of \([0,1]\) would give a convergent
sequence \(t_n\) with \(C(t_n)\to\infty\), contradicting the preceding argument.
\end{proof}

\subsection{Correctors}

\begin{proof}[Proof of Theorem~\ref{thm:correctors}]
Fix $\xi\in\R^2$. We first construct the periodic part. Consider the periodic cell problem
\begin{equation}
 L_{\per}w_{\xi}^{\per}=-\bper\cdot\xi
 \quad\text{on }\T^2,
 \qquad \avg{w_{\xi}^{\per}}=0.
 \label{corrector-periodic-cell}
\end{equation}
According to the periodic centering assumption, we have 
\[
 \int_{\T^2}\mper(x)(-\bper(x)\cdot\xi)\,dx
 =-\xi\cdot\int_{\T^2}\mper\,\bper\,dx=0.
\]
Hence the Fredholm compatibility condition is satisfied. Therefore, the cell problem \eqref{corrector-periodic-cell} admits a unique periodic solution \(w_{\xi}^{\per}\) with zero average. Since
$\aper,\bper\in C^{0,\tau}_{\per}(\R^2)$, periodic Schauder theory, based on \cite[Ch.~6, Theorem~6.2]{GT} and the Fredholm alternative \cite[Ch.~5, Theorem~5.3]{GT} gives
\[
 w_{\xi}^{\per}\in C^{2,\tau'}_{\per}(\R^2)
\]
for some $\tau'\in(0,\tau]$. In particular,
\begin{equation}
 \nabla w_{\xi}^{\per} \in (L^\infty(\R^2))^2,\quad\ D^2w_{\xi}^{\per}\in (L^\infty(\R^2))^{2\times 2}.
 \label{periodic-corrector-bounded}
\end{equation}
Define
\begin{equation}
 f_\xi:=\arme:D^2w_{\xi}^{\per}
 -\brme\cdot(\xi+\nabla w_{\xi}^{\per}).
 \label{defect-corrector-source}
\end{equation}
By \eqref{periodic-corrector-bounded} and by $a_{ij}^{\rm e},b_j^{\rm e}\in L^M(\R^2)\cap L^\infty(\R^2)$, we have
\[
 f_\xi\in L^M(\R^2)\cap L^\infty(\R^2).
\]
For $M^*>M$ and any $g\in L^M(\R^2)\cap L^\infty(\R^2)$, we have
\[
 \|g\|_{L^{M^*}(\R^2)}^{M^*}
 =\int_{\R^2}|g|^{M^*}
 \leq \|g\|_{L^\infty(\R^2)}^{M^*-M}\int_{\R^2}|g|^M=\|g\|_{L^\infty(\R^2)}^{M^*-M}\|g\|_{L^M(\R^2)}^{M}<\infty,
\]
and $g\in L^{M^*}(\R^2)$.
Therefore,
\[
 L^M(\R^2)\cap L^\infty(\R^2)\subset L^{M^*}(\R^2)\quad\text{for}\quad M^*>M,
\]
and
\begin{equation}
 f_\xi\in L^M(\R^2)\cap L^{M^*}(\R^2).
 \label{fxi-integrability}
\end{equation}

Apply Theorem~\ref{thm:central} with $q=M$. Since
\[
 \frac1{M^*}=\frac1M-\frac12,
\]
Theorem~\ref{thm:central} gives a solution $w^{\e}_\xi$, unique modulo constants, of
\begin{equation}
 L_1w^{\e}_\xi=f_\xi\quad\text{in }\R^2,
 \label{defect-corrector-equation}
\end{equation}
such that
\begin{equation}
 \nabla w^{\e}_\xi\in(L^{M^*}(\R^2))^2,\,\quad\ D^2w^{\e}_\xi\in (L^{M^*}(\R^2))^{2\times 2}.
 \label{defect-corrector-estimate}
\end{equation}
Fix the additive constant by choosing the continuous representative and imposing, for instance,
\begin{equation}
 w^{\e}_\xi(0)=0.
 \label{defect-corrector-normalization}
\end{equation}
Set
\[
 w_\xi=w_{\xi}^{\per}+w^{\e}_\xi.
\]
We now verify the corrector equation. Since
\[
 L_1=L_{\per}-\arme:D^2+\brme\cdot\nabla,
\]
we have
\[
\begin{aligned}
 L_1w_{\xi}^{\per}
 &=L_{\per}w_{\xi}^{\per}-\arme:D^2w_{\xi}^{\per}+\brme\cdot\nabla w_{\xi}^{\per} \\
 &=-\bper\cdot\xi-\arme:D^2w_{\xi}^{\per}+\brme\cdot\nabla w_{\xi}^{\per}.
\end{aligned}
\]
Using \eqref{defect-corrector-equation} and the definition \eqref{defect-corrector-source} of
$f_\xi$, we obtain
\[
\begin{aligned}
 L_1w_\xi
 &=L_1w_{\xi}^{\per}+L_1w^{\e}_\xi \\
 &=-\bper\cdot\xi-\arme:D^2w_{\xi}^{\per}+\brme\cdot\nabla w_{\xi}^{\per}
   +\arme:D^2w_{\xi}^{\per}-\brme\cdot(\xi+\nabla w_{\xi}^{\per}) \\
 &=-\bper\cdot\xi-\brme\cdot\xi=-b\cdot\xi .
\end{aligned}
\]
Therefore
\begin{equation}
 L_1w_\xi=-b\cdot\xi\quad\text{in }\R^2.
 \label{full-corrector-equation}
\end{equation}
It remains to prove the sublinearity of $w^{\e}_\xi$. Since $M^*>2$, Morrey--Poincare gives, for
every $R\geq1$,
\[
 \osc_{B_R} w^{\e}_\xi
 \leq CR^{1-2/M^*}\|\nabla w^{\e}_\xi\|_{(L^{M^*}(B_R))^2}.
\]
Using the global bound $\nabla w^{\e}_\xi\in (L^{M^*}(\R^2))^2$, we get
\begin{equation}
 \osc_{B_R}w^{\e}_\xi
 \leq CR^{1-2/M^*}\|\nabla w^{\e}_\xi\|_{(L^{M^*}(B_R))^2}.
 \label{defect-corrector-growth}
\end{equation}
By the normalization $w^{\e}_\xi(0)=0$, for $x\in B_R$,
\[
 |w^{\e}_\xi(x)|\leq \osc_{B_R} w^{\e}_\xi.
\]
Taking $R=1+|x|$ in \eqref{defect-corrector-growth}, we obtain
\[
 |w^{\e}_\xi(x)|
 \leq C(1+|x|)^{1-2/M^*}\|\nabla w^{\e}_\xi\|_{(L^{M^*}(B_R))^2}.
\]
Consequently,
\[
 \frac{|w^{\e}_\xi(x)|}{1+|x|}
 \leq C(1+|x|)^{-2/M^*}\|\nabla w^{\e}_\xi\|_{(L^{M^*}(B_R))^2}\to0
 \quad\text{as }|x|\to\infty.
\]
Hence
\begin{equation}
 \frac{w^{\e}_\xi(x)}{1+|x|}\to0
 \quad\text{as }|x|\to\infty.
 \label{defect-corrector-sublinear}
\end{equation}
Combining \eqref{defect-corrector-estimate}, \eqref{full-corrector-equation}, and
\eqref{defect-corrector-sublinear}, we obtain the claimed decomposition
\[
 w_\xi=w_{\xi}^{\per}+w^{\e}_\xi,
\]
with
\[
 L_1w_\xi=-b\cdot\xi,\,
 \quad \nabla w^{\e}_\xi\in(L^{M^*}(\R^2))^2,\,\quad\ D^2w^{\e}_\xi\in (L^{M^*}(\R^2))^{2\times 2},
 \qquad
 \frac{w^{\e}_\xi(x)}{1+|x|}\xrightarrow{|x|\to\infty}0.
\]
This proves Theorem~\ref{thm:correctors}.
\end{proof}

\section{The invariant measure}

This section proves Theorem~\ref{thm:inv}. The construction follows the duality argument of Blanc--Le Bris--Lions \cite[Section~3]{BLL}, but the adjoint estimates are written in a quantitative form suited to the two-dimensional proof.

\begin{lemma}[Adjoint local estimates]\label{lem:adjlocal}
Let \(\mu\in L^1(B_2)\) be a signed density satisfying
\begin{equation}\label{adjointlocal}
 -\partial_i\bigl(\partial_j(a_{ij}\mu)+b_i\mu\bigr)=0
 \quad\text{in }\mathcal D'(B_2),
\end{equation}
where \(a=a^T\) is uniformly elliptic and \(a_{ij},b_i\in C^{0,\tau}(B_2)\cap L^\infty(B_2)\). Then there exist \(\theta\in(0,1)\) and \(C\) such that
\begin{equation}\label{adjlocalestimate}
 \|\mu\|_{L^\infty(B_1)}+[\mu]_{C^{0,\theta}(B_1)}
 \le C\|\mu\|_{L^1(B_2)}.
\end{equation}
If \(\mu\ge0\) and \(\mu\not\equiv0\), its continuous representative is strictly positive in the interior.
\end{lemma}

\begin{proof}
Multiplying \eqref{adjointlocal} by \(-1\) gives the adjoint Fokker--Planck equation
\(L^*_{a,-b,0}(\mu dx)=0\). Since \(a\in C^{0,\tau}\), it is VMO and Dini. Since \(b_i\in L^\infty(B_2)\), we have
\(b_i\in L^{q_0}(B_2)\) for every \(q_0>2\).  The initial density estimate of
Bogachev--Krylov--R\"ockner \cite[Theorem~2.1(ii), Corollary~2.2 and Remark~2.4]{BKR2001} gives,
for \(B_{15/8}\Subset B_2\) and every \(1<p_0<2\),
\[
 \|\mu\|_{L^{p_0}(B_{15/8})}\le C_{p_0}\|\mu\|_{L^1(B_2)}.
\]
The factor \(C_{p_0}\) appearing in the estimate depends linearly on
\(1+\|b\|_{(L^{q_0}(B_2))^2}\). Bogachev--Shaposhnikov's local bootstrap
\cite[Theorem~2.1]{BogachevShaposhnikov} raises exponents according to
\(1/p_1=1/p_0+1/q_0-1/2\). Since the last two terms have negative sum for \(q_0>2\), a finite
iteration over nested balls gives
\[
 \|\mu\|_{L^p(B_{7/4})}\le C_p\|\mu\|_{L^1(B_2)},
 \qquad 1<p<\infty.
\]
Choose \(p>\max\{2,2/\tau\}\). The representation formula of Bogachev--Shaposhnikov
\cite[Proposition~3.2, Eq.~(3.3)]{BogachevShaposhnikov}, with the two-dimensional logarithmic
frozen-coefficient fundamental solution, writes \(\mu\) on \(B_1\) as a sum of potential terms. The
cut-off terms are smooth. The drift term is a first-order potential of \(b\mu\), hence is
\(C^{0,\alpha}\) for every \(\alpha<1-2/p\). The coefficient-oscillation term has kernel bounded by
\(|x-z|^{-2+\tau}\); the near-field/far-field estimate in the proof of
\cite[Theorem~3.1]{BogachevShaposhnikov}, equivalently the potential estimates in
\cite[Ch.~7, Lemma~7.12 and Theorem~7.19]{GT}, gives H\"older control for every
\(\theta<\tau-2/p\). Taking \(\theta<\min\{1-2/p,\tau-2/p\}\) gives \eqref{adjlocalestimate}. The
positivity assertion follows from \cite[Corollary~3.6]{BogachevShaposhnikov}.
\end{proof}

\begin{proof}[Proof of Theorem~\ref{thm:inv}]
Let \(X=L^q(\mathbb R^2)\cap L^{q^*}(\mathbb R^2)\), where \(q=(M^*)'\) and \(q^*=M'\). For \(f\in C_c^\infty(\mathbb R^2)\), let \(u_f\) be the solution of \(L_1u_f=f\) given by the central estimate\eqref{centralestimate}. Define
\[
\Lambda(f)
=
\int_{\mathbb R^2}
m^{\rm per}
\left(\arme:D^2u_f-\brme\cdot\nabla u_f\right).
\]
Since \(\arme,\brme\in L^M\) and \(M'=q^*\), H\"older's inequality and the central estimate \eqref{centralestimate} yield
\[
|\Lambda(f)|
\le
C\left(
\|D^2u_f\|_{L^{q^*}}
+
\|\nabla u_f\|_{L^{q^*}}
\right)
\le
C\|f\|_X.
\]
Hence \(\Lambda\) extends uniquely to a bounded linear functional on \(X\). By \(X^*=L^{q'}+L^{(q^*)'}=L^{M^*}+L^M\), there exists \(m^{\e}\in L^{M^*}+L^M\) such that
\[
\int_{\mathbb R^2}m^{\e} f=\Lambda(f)
\]
for all \(f\in X\). Taking \(f=L_1\varphi\), with \(\varphi\in C_c^\infty(\mathbb R^2)\), and using uniqueness modulo constants for \(L_1u=f\), we obtain \(u_f=\varphi+\mathrm{const}\). Hence
\[
\int_{\mathbb R^2}m^{\e}L_1\varphi
=
\int_{\mathbb R^2}
m^{\rm per}
\left(\arme:D^2\varphi-\brme\cdot\nabla\varphi\right).
\]
Therefore
\[
\int_{\mathbb R^2}(m^{\rm per}+m^{\e})L_1\varphi
=
\int_{\mathbb R^2}m^{\rm per}L_{\rm per}\varphi
=
0.
\]
Thus \(m=m^{\rm per}+m^{\e}\) satisfies the full adjoint equation.

Lemma~\ref{lem:adjlocal} gives, uniformly in \(x\),
\begin{equation}\label{munifholder}
 \|m\|_{L^\infty(B_1(x))}+[m]_{C^{0,\theta}(B_1(x))}
 \le C\|m\|_{L^1(B_2(x))}.
\end{equation}
The right side is uniformly bounded since \(\mper\in L^\infty(\R^2)\) and \(m^{\e}\in L^{M^*}+L^M\). Thus \(m^{\e}\in L^\infty\cap C^{0,\theta}_{\unif}\). If \(m^{\e}=g+h\), with \(g\in L^{M^*}\) and \(h\in L^M\), then on \(\{|m^{\e}|\le2|g|\}\) the \(L^{M^*}\) bound follows from \(g\), while on the complementary set \(|m^{\e}|\le2|h|\) and boundedness gives \(|m^{\e}|^{M^*}\le C|h|^M\). Hence \(m^{\e}\in L^{M^*}\). Uniform H\"older continuity and \(L^{M^*}\) integrability implies \(m^{\e}(x)\to0\) uniformly as \(|x|\to\infty\).

Since \(\mper\ge c_{\per}>0\), for large \(R\) one has \(m\ge c_{\per}/2\) on
\(\R^2\setminus B_R\). Set \(\mathcal L=-L_1=a:D^2-b\cdot\nabla\), so that
\(\mathcal L^*m=0\). Choose a ball \(\mathcal B\) with
\(\overline{B_R}\Subset\mathcal B\), a pole \(A_0\in\mathcal B\setminus\overline{B_R}\), and let
\(g(y)=G^{\mathcal L}_{\mathcal B}(A_0,y)\) be Bauman's positive normalizing Green function. Since the pole
lies outside \(\overline{B_R}\), \(g\) is bounded above and below by positive constants there.
Thus \(-m/g\) is a normalized adjoint solution in \(B_R\) with nonpositive boundary values.
Bauman's maximum principle \cite[Proposition~4.1 and Theorem~4.2]{Bauman} gives \(-m/g\le0\), and hence
\(m\ge0\) in \(B_R\). Since \(m\ge0\) on \(\R^2\) and \(m\not\equiv0\), applying the positivity
assertion of Lemma~\ref{lem:adjlocal} along a finite chain of overlapping balls gives \(m>0\)
everywhere. The positive lower bound outside
\(B_R\) and compactness of \(\overline{B_R}\) yield \(\inf_{\R^2}m>0\).
\end{proof}

\begin{corollary}[Uniqueness in the decaying class]\label{cor:unique}
The invariant density is unique among \(m=\mper+m^{\e}\) with \(m^{\e}\in L^{M^*}(\R^2)\cap L^\infty(\R^2)\) and \(\sup_{|x|\ge R}|m^{\e}(x)|\to0\).
\end{corollary}

\begin{proof}
If \(m^{(1)}\) and \(m^{(2)}\) are two such densities, set \(\widetilde{m}=m^{(2)}-m^{(1)}\). Then \(\widetilde{m}\) is an adjoint solution, and \(\widetilde{m}\to0\) uniformly at infinity. Applying the preceding positivity argument to \(m^{(1)}\) gives \(0<c_1\le m^{(1)}\le C_1\). For each \(\eps>0\), choose \(R_\eps\ge\eps^{-1}\) so large that \(|\widetilde{m}|\le\eps\) on \(\partial B_{R_\eps}\). The adjoint solutions
\[
 \rho_\pm=\frac{\eps}{c_1}m^{(1)}\pm \widetilde{m}
\]
are nonnegative on \(\partial B_{R_\eps}\). Choose an ambient ball, an exterior pole and the positive
Green weight \(g_{R_\eps}\) as above. Bauman's maximum principle \cite[Proposition~4.1 and Theorem~4.2]{Bauman}, applied
to \(-\rho_\pm/g_{R_\eps}\), gives \(\rho_\pm\ge0\) in \(B_{R_\eps}\). Every fixed point belongs to
\(B_{R_\eps}\) for all sufficiently small \(\eps\); letting \(\eps\downarrow0\) gives
\(\widetilde{m}=0\).
\end{proof}

\section{Divergence-form reduction}\label{sec:reduction}
\subsection{A two-dimensional Hodge estimate}\label{sec:aux-tools}

\begin{lemma}[Two-dimensional stream-function estimate]\label{lem:hodge}
Let \(1<\eta<2<\gamma<\infty\), let \(\theta\in(0,1)\), and set
\(\eta^*=2\eta/(2-\eta)\). Let
\[
 F=H+\diverg Q,
 \qquad \diverg F=0,
\]
where
\[
 H\in (L^\eta(\R^2))^2\cap (L^\gamma(\R^2))^2,
 \qquad
 Q\in (L^{\eta^*}(\R^2))^{2\times2}\cap (L^\infty(\R^2))^{2\times2}\cap (C^{0,\theta}_{\unif}(\R^2))^{2\times2}.
\]
Then there is a skew-symmetric matrix field \(B=\psi R\), with
\(R=\begin{pmatrix}0&1\\ -1&0\end{pmatrix}\), such that \(\diverg B=F\) in \(\mathcal D'(\R^2)\) and
\begin{equation}\label{hodge-estimate}
 \|B\|_{((L^{\eta^*}\cap L^\infty)(\R^2))^{2\times 2}}
 \le C\Big(\|H\|_{((L^\eta\cap L^\gamma)(\R^2))^2}+\|Q\|_{((L^{\eta^*}\cap L^\infty)(\R^2))^{2\times 2}}
       +[Q]_{C^{0,\theta}_{\unif}}\Big).
\end{equation}
Moreover, for every \(0<\beta<\min\{\theta,1-2/\gamma\}\),
\begin{equation}\label{hodge-holder-estimate}
 [B]_{C^{0,\beta}_{\unif}}
 \le C_\beta\Big(\|H\|_{((L^\eta\cap L^\gamma)(\R^2))^2}
 +\|Q\|_{((L^{\eta^*}\cap L^\infty)(\R^2))^{2\times 2}}
 +[Q]_{C^{0,\theta}_{\unif}}\Big).
\end{equation}
\end{lemma}

\begin{proof}
We first take \(H_i,\,Q_{ij}\in C_c^\infty(\R^2)\). Since \(F\) is solenoidal, define the Fourier function of a stream function $\psi$ by
\[
 \widehat\psi(\xi)=\frac{\xi_1\widehat F_2(\xi)-\xi_2\widehat F_1(\xi)}{i|\xi|^2},
 \qquad \xi\ne0 .
 \label{stream-fourier}
\]
Then \((-\partial_2\psi,\partial_1\psi)=F\) in \(\mathcal S'(\R^2)\), because
\(\xi\cdot \widehat F(\xi)=0\). Thus \(B=\psi R\) satisfies \(\diverg B=F\). We decompose \(\psi=\psi_H+\psi_Q\) according to \(F=H+\diverg Q\).

Let \(\Gamma(x)=-(2\pi)^{-1}\log |x|\). The part generated by \(H\) is a finite sum of first-order potentials, namely
\[
 \psi_H=-\partial_1\Gamma*H_2+\partial_2\Gamma*H_1 .
\]
Since \(|\nabla\Gamma(x)|\lesssim |x|^{-1}\), the Hardy--Littlewood--Sobolev inequality in \(\R^2\) gives \(\|\psi_H\|_{L^{\eta^*}}\le C\|H\|_{L^\eta}\). The \(L^\infty\) bound follows by splitting the convolution into \(|z|<1\) and \(|z|\ge1\). Since \(|z|^{-1}\in L^{\gamma'}(B_1)\) with $\gamma'=\gamma/(\gamma-1)$ and \(|z|^{-1}\mathbf1_{\{|z|\ge1\}}\in L^{\eta'}(\R^2)\) with $\eta'=\eta/(\eta-1)>2$, the near field is controlled with the exponent \(\gamma>2\) and the far field with the exponent \(\eta<2\), respectively. Hence
\begin{equation}\label{psiH-est}
 \|\psi_H\|_{(L^{\eta^*}\cap L^\infty)(\R^2)}
 \le C\bigl(\|H\|_{((L^\eta\cap L^\gamma)(\R^2))^2}\bigr).
\end{equation}

We next consider the \(\diverg Q\)-part. Since \(\Delta\psi_Q=\curl(\diverg Q)\), \(\psi_Q\) is a finite sum of second-order singular integrals of the entries of \(Q\), up to the standard zero-order terms coming from the distributional identity for \(\partial_{\alpha\beta}\Gamma\). For \(z\ne0\),
\(\partial_{\alpha\beta}\Gamma(z)=-(2\pi)^{-1}(\delta_{\alpha\beta}-2\omega_\alpha\omega_\beta)|z|^{-2}\), \(\omega=z/|z|\). Thus the principal-value kernels are homogeneous kernels \(K(z)=\Omega(z/|z|)|z|^{-2}\), with \(\Omega\in C^\infty(S^1)\), even, and of mean zero. Grafakos' homogeneous singular-integral theorem \cite[Theorem~5.2.10]{Grafakos} gives
\begin{equation}\label{psiQ-Lp-est}
 \|\psi_Q\|_{L^{\eta^*}(\R^2)}
 \le C\|Q\|_{(L^{\eta^*}(\R^2))^{2\times2}}.
\end{equation}
The possible delta terms in \(\partial_{\alpha\beta}\Gamma=\operatorname{p.v.}\partial_{\alpha\beta}\Gamma+c_{\alpha\beta}\delta_0\) only produce bounded zero-order terms \(c_{\alpha\beta}Q\).

For the \(L^\infty\) estimate, write each singular-integral term as
\[
 TQ(x)=c_0Q(x)+\operatorname{p.v.}\int_{|z|<1}K(z)\bigl(Q(x-z)-Q(x)\bigr)\,dz
       +\int_{|z|\ge1}K(z)Q(x-z)\,dz .
 \label{TQ-decomp}
\]
The first term is bounded by \(\|Q\|_{(L^\infty(\R^2)})^{2\times 2}\). Since \(|K(z)|\le C|z|^{-2}\), the local principal value is bounded by \(C[Q]_{C^{0,\theta}_{\unif}}\). For the far field, \(\eta^*>2\) gives \(|z|^{-2}\mathbf1_{\{|z|\ge1\}}\in L^{(\eta^*)'}(\R^2)\), and H\"older's inequality gives a bound by \(C\|Q\|_{(L^{\eta^*}(\R^2))^{2\times 2}}\). Consequently
\begin{equation}\label{psiQ-Linfty-est}
 \|\psi_Q\|_{L^\infty(\R^2)}
 \le C\bigl(\|Q\|_{L^\infty(\R^2)}+[Q]_{C^{0,\theta}_{\unif}}+\|Q\|_{(L^{\eta^*}(\R^2))^{2\times 2}}\bigr).
\end{equation}
To obtain the H\"older estimate, first observe that \(\nabla\psi_H\) is a finite sum of
Calder\'on--Zygmund operators applied to \(H\). Hence
\(\|\nabla\psi_H\|_{L^\gamma}\le C\|H\|_{L^\gamma}\), and Morrey's inequality gives
\begin{equation}\label{psiH-holder-est}
 [\psi_H]_{C^{0,1-2/\gamma}_{\unif}}\le C\|H\|_{L^\gamma}.
\end{equation}
For a singular-integral term \(TQ\), fix \(x_0\in\R^2\) and choose a translated cutoff
\(\chi_{x_0}\) which equals one on \(B_2(x_0)\) and vanishes outside \(B_3(x_0)\).
The standard H\"older estimate for smooth homogeneous Calder\'on--Zygmund kernels gives
\[
 [T(\chi_{x_0}Q)]_{C^{0,\theta}(B_1(x_0))}
 \le C\bigl(\|Q\|_{L^\infty}+[Q]_{C^{0,\theta}_{\unif}}\bigr).
\]
On the other hand, \(|\nabla K(z)|\le C|z|^{-3}\), so H\"older's inequality yields
\[
 \|\nabla T((1-\chi_{x_0})Q)\|_{L^\infty(B_1(x_0))}
 \le C\|Q\|_{L^{\eta^*}},
\]
with a constant independent of \(x_0\). The zero-order terms obey the same local H\"older
bound. Thus
\begin{equation}\label{psiQ-holder-est}
 [\psi_Q]_{C^{0,\theta}_{\unif}}
 \le C\bigl(\|Q\|_{L^{\eta^*}}+\|Q\|_{L^\infty}+[Q]_{C^{0,\theta}_{\unif}}\bigr).
\end{equation}
Combining \eqref{psiH-est}, \eqref{psiQ-Lp-est}, \eqref{psiQ-Linfty-est},
\eqref{psiH-holder-est} and \eqref{psiQ-holder-est} proves the estimates for smooth compactly
supported data.

For general data, define \(\psi_H\) by the first-order potentials above and define the second-order potentials, in the principal-value sense including the zero-order terms, by
\[
 \psi_Q=-\sum_{i=1}^2\partial_1\partial_i\Gamma*Q_{i2}
        +\sum_{i=1}^2\partial_2\partial_i\Gamma*Q_{i1},
 \qquad \psi=\psi_H+\psi_Q.
\]
The preceding estimates apply directly, so \(\psi\in L^{\eta^*}(\R^2)\cap L^\infty(\R^2)\cap C^{0,\beta}_{\unif}(\R^2)\), and \eqref{hodge-estimate}--\eqref{hodge-holder-estimate} hold for \(B=\psi R\). Approximating \(H\) in \(L^\eta\) and \(Q\) in \(L^{\eta^*}\) by smooth compactly supported fields, the continuity of the Hardy--Littlewood--Sobolev and Calder\'on--Zygmund operators yields
\[
 \Delta\psi=\curl F
 \quad\hbox{in }\mathcal D'(\R^2).
\]
Let \(G=F-\diverg(\psi R)\). Since \(\diverg F=0\), we have \(\diverg G=\curl G=0\), and hence \(\Delta G=0\) in distributions. Moreover, with \(S=Q-\psi R\in (L^{\eta^*}(\R^2))^{2\times2}\),
\[
 G=H+\diverg S.
\]
For a standard mollifier \(\rho_\varepsilon\), the entire harmonic field \(G*\rho_\varepsilon=h_\varepsilon+g_\varepsilon\) has \(h_\varepsilon=H*\rho_\varepsilon\in L^\eta\) and \(g_\varepsilon=(\diverg S)*\rho_\varepsilon\in L^{\eta^*}\). The mean-value property therefore gives, for every fixed \(x\),
\[
 |(G*\rho_\varepsilon)(x)|
 \le C R^{-2/\eta}\|h_\varepsilon\|_{L^\eta}
      +C R^{-2/\eta^*}\|g_\varepsilon\|_{L^{\eta^*}}
 \longrightarrow0
 \qquad(R\to\infty).
\]
Thus \(G*\rho_\varepsilon=0\) for every \(\varepsilon>0\), and letting \(\varepsilon\downarrow0\) gives \(G=0\). Hence \(\diverg B=F\), completing the proof.
\end{proof}

\subsection{Harmonic-coordinate Hodge construction and Piola pull-back}\label{sec:transformed}

We now prove Theorem~\ref{thm:divred}. We take an arbitrary invariant measure satisfying the hypotheses of the theorem, push it to harmonic coordinates, construct the skew potential there, and pull it back by the Piola transform.

\begin{proof}[Proof of Theorem~\ref{thm:divred}]
Let \(\Theta(Y)=P^{-1}(Y)\) and \(J(x)=\det DP(x)\). We first record the periodic objects in harmonic coordinates. The following formulas are pointwise definitions after choosing the continuous representatives; all divergence identities are distributional. Periodicity follows from \(P(x+k)=P(x)+k\) and \(\Theta(Y+k)=\Theta(Y)+k\) with $Y=P(x)$ and $k\in\mathbb{Z}^2$. Set
\begin{align}
 \widehat m^{\per}(Y)&=\frac{\mper(\Theta(Y))}{J(\Theta(Y))},\label{hmper}\\
 \widehat B^{\per}(Y)&=\frac1{J(\Theta(Y))}DP(\Theta(Y))B^{\per}(\Theta(Y))DP(\Theta(Y))^T,\label{hBper}\\
 \widehat A^{\mathrm{div}}_{\per}(Y)&=\frac1{J(\Theta(Y))}DP(\Theta(Y))\Adiv_{\per}(\Theta(Y))DP(\Theta(Y))^T .\label{hAdivper}
\end{align}
The transformed periodic coefficients are
\(\haper(Y)=DP(\Theta(Y))\aper(\Theta(Y))DP(\Theta(Y))^T\) and \(\widehat b_{\per}=0\); moreover \(\haper\) is periodic, H\"older continuous, symmetric and uniformly elliptic, while \(\hae_{ij},\hbe_{i}\in L^M(\R^2)\cap L^\infty(\R^2)\), with uniform H\"older regularity after lowering the exponent if necessary. These assertions follow from the formulas in Theorem~\ref{thm:gauge}, the boundedness of \(DP,D^2P,D\Theta\), and the bi-Lipschitz change of variables.

For \(u(x)=w(P(x))\), the periodic identity gives
\[
 \mper L_{\per}u=-\diverg_x(\Adiv_{\per}\nabla_xu).
\]
Using \(\nabla_xu=DP^T\nabla_Yw\) and the Piola identity
\begin{equation}\label{piola-identity}
 \diverg_x\bigl(JDP^{-1}F(P(x))\bigr)=J\diverg_YF(P(x)).
\end{equation}
we obtain
\begin{equation}\label{periodic-hat-div}
 \widehat m^{\per}\widehat L_{\per}w
 =-\diverg_Y(\widehat A^{\mathrm{div}}_{\per}\nabla_Yw),
 \qquad
 \widehat A^{\mathrm{div}}_{\per}=\widehat m^{\per}\haper-\widehat B^{\per}.
\end{equation}
The identity \eqref{piola-identity} is the usual Piola formula and follows by direct differentiation for smooth fields, then by approximation. Since in two dimensions \(VRV^T=(\det V)R\) holds for any $V\in \R^{2\times 2}$ and $\begin{pmatrix}0&1\\ -1&0\end{pmatrix}$, the matrix \(\widehat B^{\per}\) is skew-symmetric whenever \(B^{\per}\) is skew-symmetric.

Let \(m=\mper+m^{\e}\) be a positive invariant density satisfying the hypotheses of the theorem. Define its push-forward density by
\begin{equation}\label{pushforwardm}
 \widehat m(Y)=\frac{m(\Theta(Y))}{J(\Theta(Y))}.
\end{equation}
For \(\zeta\in C_c^\infty(\R^2)\), put \(\phi=\zeta\circ P\). Then \(\phi\in C_c^2(\R^2)\), and the invariant-measure identity extends from \(C_c^\infty\) to such tests by approximation in the \(C^2\)-norm on the common compact support. The transformation formula gives \(L\phi(x)=\widehat L\zeta(P(x))\), hence \(\int \widehat m\,\widehat L\zeta=0\) follows from \(\int m L\phi=0\). Thus \(\widehat m\) is an invariant density for the transformed operator. Moreover
\begin{equation}\label{hme-decomp}
 \widehat m=\widehat m^{\per}+\widehat m^{\e},
 \qquad
 \widehat m^{\e}(Y)=\frac{m^{\e}(\Theta(Y))}{J(\Theta(Y))}.
\end{equation}
so \(\widehat m^{\e}\in L^{M^*}(\R^2)\cap L^\infty(\R^2)\), and \(\widehat m\) is bounded above and below.

We next record the regularity needed for the Hodge estimate. Since \(m\) solves the full adjoint equation and \(m,m^{\e}\in L^\infty_{\loc}\), Lemma~\ref{lem:adjlocal} applied on unit balls gives \(m\in C^{0,\theta_m}_{\unif}\). As \(\mper\in C^{0,\theta_0}_{\per}\), also \(m^{\e}\in C^{0,\bar\theta}_{\unif}\). Composition with \(\Theta\) and multiplication by \((J\circ \Theta)^{-1}\) preserve uniform H\"older continuity, after reducing the exponent; hence
\begin{equation}\label{hatmholder}
 \widehat m,\ \widehat m^{\e}\in C^{0,\theta_1}_{\unif}(\R^2).
\end{equation}
for some \(\theta_1>0\). The same applies to \(\haper\) and \(\hae\).

Because \(\widehat b^{\per}=0\), the defect part of the transformed source is
\begin{equation}\label{hodge-source}
 \widehat F^{\e}
 =\widehat m\,\hbe+
  \diverg\bigl(\widehat m^{\e}\haper+\widehat m\hae\bigr).
\end{equation}
The transformed invariant equation and the periodic identity \eqref{periodic-hat-div} give \(\diverg \widehat F^{\e}=0\) in distributions. Write \(\widehat F^{\e}=H+\diverg Q_H\), where \(H=\widehat m\hbe\) and \(Q_H=\widehat m^{\e}\haper+\widehat m\hae\). Since \(\widehat m\in L^\infty(\R^2)\) and \(\hbe_i\in L^M(\R^2)\cap L^\infty(\R^2)\), one has \(H_i\in L^M(\R^2)\cap L^\gamma(\R^2)\) for every finite \(\gamma>2\). Also \(Q_H\in (L^{M^*}(\R^2))^{2\times 2}\cap (L^\infty(\R^2))^{2\times 2}\cap (C^{0,\theta_1}_{\unif}(\R^2))^{2\times 2}\): the first summand is controlled by \(\widehat m^{\e}\in L^{M^*}(\R^2)\cap L^\infty(\R^2)\), while the second belongs to \((L^M(\R^2))^2\cap (L^\infty(\R^2))^2\subset (L^{M^*}(\R^2))^2\cap (L^\infty(\R^2))^2\). The H\"older bound follows from \eqref{hatmholder} and the uniform H\"older bounds for the coefficients.

Choose \(0<\beta<\theta_1\) and then \(\gamma>2/(1-\beta)\). Lemma~\ref{lem:hodge},
used with \(\eta=M\), gives a skew-symmetric matrix \(\widehat B^{\e}\) such that
\begin{equation}\label{hatBe}
 \widehat B_{ij}^{\e}\in L^{M^*}(\R^2)\cap L^\infty(\R^2)\cap C^{0,\beta}_{\unif}(\R^2),
 \qquad
 \diverg\widehat B^{\e}=\widehat F^{\e}.
\end{equation}
Adding the periodic part in \eqref{hBper}, set \(\widehat B=\widehat B^{\per}+\widehat B^{\e}\). Then
\begin{equation}\label{hatB-source}
 \diverg\widehat B=\widehat m\,\widehat b+
 \diverg(\widehat m\,\widehat a),
 \qquad
 \widehat A=\widehat m\,\widehat a-\widehat B .
\end{equation}
and therefore
\begin{equation}\label{hatdivid}
 \widehat m\,\widehat Lw=-\diverg_Y(\widehat A\nabla_Yw).
\end{equation}
(This formula is distributional; testing against compactly supported functions gives the stated identity.) Pull back \(\widehat A\) and \(\widehat B\) by the Piola formula:
\begin{equation}\label{piolapullAB}
 A(x)=J(x)DP(x)^{-1}\widehat A(P(x))DP(x)^{-T},
 \qquad
 B(x)=J(x)DP(x)^{-1}\widehat B(P(x))DP(x)^{-T}.
\end{equation}
Using \eqref{piola-identity} in \eqref{hatdivid} yields \(mLu=-\diverg_x(A\nabla u)\). Since \(\widehat a(P(x))=DP(x)a(x)DP(x)^T\), the first formula in \eqref{piolapullAB} gives
\[
 A=m\,a-B.
\]
The two-dimensional identity \(\det V\,V^{-1}RV^{-T}=R\) shows that the Piola pull-back preserves skew-symmetry. From \eqref{hBper} the periodic part pulls back exactly to \(B^{\per}\); hence \(B=B^{\per}+B^{\e}\), where
\[
 B^{\e}=JDP^{-1}\widehat B^{\e}(P)DP^{-T}=\widehat\psi^{\e}(P)R
 \in (L^{M^*}(\R^2))^{2\times 2}\cap (L^\infty(\R^2))^{2\times 2}
 \cap (C^{0,\beta}_{\unif}(\R^2))^{2\times 2}.
\]
Here \(\widehat B^{\e}=\widehat\psi^{\e}R\), and the last equality follows from the same
two-dimensional identity. Since
\[
 A^{\e}=m^{\e}\aper+m\arme-B^{\e},
\]
the uniform H\"older bounds for \(m,m^{\e},\aper,\arme\) give
\(A^{\e}\in (C^{0,\beta}_{\unif}(\R^2))^{2\times 2}\), after decreasing \(\beta\) if
necessary. The preceding integrability argument also gives
\(A^{\e}\in (L^{M^*}(\R^2))^{2\times 2}\cap(L^\infty(\R^2))^{2\times 2}\).
Any function in \(L^{M^*}(\R^2)\cap C^{0,\beta}_{\unif}(\R^2)\) tends to zero uniformly at
infinity: otherwise uniform H\"older continuity would produce infinitely many disjoint balls on
which its absolute value has a fixed positive lower bound. Hence \(A^{\e}(x),B^{\e}(x)\to0\)
uniformly as \(|x|\to\infty\). Finally, \(\xi\cdot A(x)\xi=m(x)\xi\cdot a(x)\xi\), because
\(B\) is skew. Since \(m\) has a positive lower bound and \(a\) is uniformly elliptic, \(A\) is coercive.
\end{proof}

\subsection{Endpoint restrictions}

The results are limited to scalar equations, $C^{0,\tau}$ regular coefficients and non-endpoint
exponents $1<q<2$, $1<r,s<2$.
The endpoint $q=1$ is not covered because the closedness step uses the finite-energy Liouville
theorem with $q^*>2$. The endpoints $r=1$ and $s=1$ are not covered. No system result is claimed.

\subsection*{Sharpness of the non-endpoint range}

The restrictions above are necessary for the estimates proved here. First, the strong endpoint $q=1$ already fails for the
Laplacian in two dimensions. If $f\in C_c^\infty(\R^2)$ and $\int_{\R^2}f\,dx\neq0$, the solution
of $-\Delta u=f$ satisfies
\[
 \nabla u(x)=\frac{\int_{\R^2}f\,dx}{2\pi}\frac{x}{|x|^2}+O(|x|^{-2})
 \qquad (|x|\to\infty),
\]
so $|\nabla u(x)|\sim |x|^{-1}$ and therefore $\nabla u\notin L^2(\R^2)$. Thus the estimate
corresponding to $q=1$, for which $q^*=2$, cannot hold in the strong form used here.

Second, the borderline $M=2$ is also critical for the Hodge step. Let
$\psi(x)=\log\log|x|$ for large $|x|$, with a smooth cutoff near the origin, and set
$F=\nabla^\perp\psi$. Then $F\in (L^2(\R^2))^2\cap (L^\infty(\R^2))^2$ because
$|\nabla\psi(x)|\simeq(|x|\log|x|)^{-1}$ at infinity, but any stream function $\varphi$ with
$\nabla^\perp\varphi=F$ differs from $\psi$ by a constant on the exterior region and is therefore
unbounded. This example shows the necessity of the restriction $1<M<2$ in the present Hodge
argument and of the additional $L^\gamma$, $\gamma>2$, input in Lemma~\ref{lem:hodge}.

Finally, defects in $L^1(\R^2)\cap L^\infty(\R^2)$ are still covered by the non-endpoint theory after choosing
any fixed $M\in(1,2)$, since interpolation gives $L^1(\R^2)\cap L^\infty(\R^2)\subset L^M(\R^2)$. If $C_M$ denotes
one of the constants produced after this choice of $M$, no assertion is made that
\[
 \sup_{1<M\le M_0} C_M<\infty
\]
for any fixed $M_0\in(1,2)$, and no endpoint estimate at $M=1$ is claimed.

\end{document}